\documentclass[11pt,english]{amsart}
\usepackage{amsmath,amssymb,amsthm}
\usepackage{hyperref}
\usepackage{float}
\usepackage{tikz}
\usepackage{graphicx}
\usetikzlibrary{calc}
\usetikzlibrary{arrows,decorations.pathmorphing,backgrounds,positioning,fit,petri}

\newtheorem{theorem}{Theorem}[section]
\newtheorem{lemma}[theorem]{Lemma}
\newtheorem{corollary}[theorem]{Corollary}

\newtheorem{proposition}[theorem]{Proposition}
\theoremstyle{definition}
\newtheorem{definition}[theorem]{Definition}
\newtheorem{remark}[theorem]{Remark}

\newtheorem{claim}[theorem]{Claim}

\newcommand{\Aut}{\hbox{\rm Aut}}

\newcommand{\inv}{^{-1}}
\renewcommand{\mod}{\hbox{mod}\, }

\newcommand{\lin}{\overline}

\newcommand{\NN}{\mathbb{N}}
\newcommand{\ZZ}{\mathbb{Z}}

\newcommand{\Cay}{\hbox{\rm Cay}}

\newcommand{\core}{\hbox{\rm core}}
\newcommand{\Cyc}{\hbox{\rm CycCov}}
\newcommand{\GC}{\hbox{\rm GenCov}}

\newcommand{\fib}{{\rm{fib}}}
\DeclareMathOperator{\lcm}{{\rm{lcm}}}

\renewcommand{\inv}{\mathop{{\rm inv}}}
\DeclareMathOperator{\Inv}{{\rm inv}}
\DeclareMathOperator{\beg}{{\rm beg}}
\DeclareMathOperator{\term}{{\rm end}}
\DeclareMathOperator{\val}{{\rm val}}

\newcommand{\ccv}{\rm ccv}

\newcommand{\D}{{\rm D}}
\newcommand{\V}{{\rm V}}

\renewcommand{\mod}{\hbox{\rm{mod }}}

\makeatletter
\numberwithin{equation}{section}
\numberwithin{figure}{section}

\makeatother

\usepackage{babel}

\tikzstyle{vertu}=[circle,draw=black!,fill=black!,thick]
\tikzstyle{vertv}=[circle,draw=black!,fill=white!,thick]
\tikzstyle{blank}=[circle,draw=white!,fill=white!,thick,inner sep = 0.11mm]
\tikzstyle{post}=[->,shorten >=1pt,>=latex,thick]
\tikzstyle{dentro}=[<-,>=latex,thick]

\title[Cubic vertex-transitive graphs]{Finite cubic graphs admitting a cyclic group of automorphism with at most three orbits on vertices}
\author{Primo\v{z} Poto\v{c}nik}
\author{Micael Toledo}

\address{Primo\v{z} Poto\v{c}nik, Faculty of Mathematics and Physics, University of Ljubljana, Jadranska 21, SI-1000 Ljubljana, Slovenia.\newline
\indent Also affiliated with: Institute of Mathematics, Physics and Mechanics, Jadranska 19, SI-1000 Ljubljana, Slovenia.
}
\email{primoz.potocnik@fmf.uni-lj.si}

\address{Micael Toledo, Institute of Mathematics, Physics and Mechanics, Jadranska 19, SI-1000.\newline
 Also affiliated with: University of Primorska, Faculty of Mathematics, Natural Sciences and Information Technologies, Glagolja\v{s}ka 8, SI-6000 Koper, Slovenia.}
\email{micael.toledo@imfm.uni-lj.si}

\thanks{The authors gratefully acknowledge support of the Slovenian Research Agency: Core Programme P1-0294, Research Project J1-1691 and the Young Researcher Scholarship programme.}

\begin{document}

\begin{abstract}
The theory of voltage graphs has become a standard tool in the study graphs admitting a semiregular group of automorphisms. We introduce the notion of a cyclic generalised voltage graph to extend the scope of this theory to graphs admitting a cyclic group of automorphism that may not be semiregular. We use this new tool to classify all cubic graphs admitting a cyclic group of automorphisms with at most three vertex-orbits and we characterise vertex-transitivity for each of these classes. In particular, we show that a cubic vertex-transitive graph admitting a cyclic group of automorphisms with at most three orbits on vertices either belongs to one of $5$ infinite families or is isomorphic to the well-know Tutte-Coxeter graph.
\end{abstract}
\maketitle

\section{Introduction}
\label{sec:intro}

Graphs $\Gamma$ whose automorphism group $\Aut(\Gamma)$ contains a cyclic subgroup $G$ with a small number of orbits on the vertex-set $\V(\Gamma)$ exhibit many interesting features and have been the focus of study in many settings. For example,
when the cyclic group $G$ has a single orbit on $\V(\Gamma)$, then the graph belongs to the family of circulant, a natural and widely studied class of graphs (see \cite{circ0} for an early mention of circulants in a famous research problem of \'Ad\'am, and \cite{circ1,circ2,circ3,circ4} for some of the most recent papers). When $G$ partitions
$\V(\Gamma)$ into $k$ orbits of equal size, the graph is then called a $k$-multicirculant. A $1$-, $2$- or $3$-multicirculant is often also called a circulant, bicirculant or tricirculant, respectively.
When studying $k$-multicirculants, typically additional graph theoretical and symmetry conditions are imposed, such as vertex-transitivity, edge-transitivity or arc-transitivity
(see \cite{metacirc,multicirc1,multicirc2,bicirc,multicirc3,multicirc4,bic,MPtricirc}, for example).
Multicirculants also feature in the famous {\em multicirculant conjecture} \cite{poly2,poly1,cubicpoly2} which states that
every finite vertex-transitive graph is a $k$-multicirculant for some $k$ smaller than the number of vertices of the graph.

The case where the orbits of $G$ are not all of equal size is far less studied and understood, which is partly due to the lack of appropriate tools to study such graphs: while 
$k$-multicirculant can be (and usually are) studied using the theory of regular covering projections (see, for example, \cite{grosstuc,lift,Sir}), 
no analogous theory  for the case of $G$-orbits of unequal size existed until recently. 

This paper has two main objectives. 
One objective is to adapt the recently developed theory of generalised covers of graphs \cite{MPgc} to the situation where a graph possesses a non-trivial cyclic group of automorphisms (possibly with orbits of unequal size).

The second objective is to use this theory in order to determine all the finite connected
cubic graphs $\Gamma$ admitting a cyclic automorphism group $G$ with at most
three vertex-orbits.
We classify these graphs into 25 families, each family arising from one of 25 possible
quotients of $\Gamma$ by $G$; see Theorem~\ref{the:main}. We then determine which of the graphs in these
families are vertex-transitive and prove the following theorem (the graphs appearing in the statement are defined in Section~\ref{sec:classification}): 

\begin{theorem}
\label{the:vt}
Let $\Gamma$ be a connected, simple, cubic  graph admitting a cyclic group of automorphisms having at most $3$ orbits on the vertices of $\Gamma$. Then $\Gamma$ is vertex-transitive if and only if it is isomorphic to the Tutte-Coxeter graph $\Gamma_{25}(10;1,3)$ or belongs to one of the following five infinite families:
\begin{enumerate}
\item $\Gamma_1(m;r)$ with $m$ even, $m \geq 4$, $\gcd(\frac{m}{2},r)=1$, $r \in \{1,2\}$;
\item $\Gamma_2(m;r,1)$ with $m \geq 3$, $r^2 \equiv \pm 1$ $(\mod m)$, or $m = 10$ and $r=2$;
\item $\Gamma_4(m;r,s)$ with $m \geq 3$, $r \neq s$, $\gcd(m,r,s)=1$;
\item $\Gamma_{22}(m;2,1)$ with $m \geq 4$ and $\frac{m}{2}$ odd;
\item $\Gamma_{23}(m;r,1)$ with $\frac{m}{2} \equiv 1$ $(\mod 4)$ and $r = (\frac{m}{2}+3)/2$,  or $\frac{m}{2} \equiv 3$ $(\mod 4)$ and $r = (\frac{3m}{2}+3)/2$, or $m=4$ and $r=0$.
\end{enumerate} 
\end{theorem}

\begin{remark}
We would like to point out that all the graphs appearing in Theorem~\ref{the:vt} are
in fact $k$-multicirculants for some $k\in \{1,2,3\}$. More precisely,
the Tutte-Coxeter graph $\Gamma_{25}(10,1,3)$ is the smallest $5$-arc-transitive cubic graph, whose automorphism group is isomorphic to $\Aut(S_6)$. It is a tricirculant, but not a bicirculant or a circulant.
Further, the graphs $\Gamma_1(m;r)$ from part (1) of the theorem are the circulants $\Cay(\ZZ_m; \{r,-r,m/2\})$ with $m$ even. The graphs
$\Gamma_2(m;r,1)$ from part (2) are the generalised Petersen graphs ${\rm GP}(m,r)$, which are all bicirculants, some of them even circulants. The graphs $\Gamma_4(m;r,s)$ from part (3) are bicirculants, in fact the cubic cyclic Haar graphs \cite{haar}, isomorphic to the Cayley graph $\Cay(D_m;\{\tau,\rho^r\tau,\rho^s\tau\})$ where $D_m = \langle \rho,\tau \mid \rho^m, \tau^2, (\rho\tau)^2\rangle$ is the dihedral group of order $2m$. The graphs $\Gamma_{22}(m;2,1)$ from part (4) are the tricirculants that were denoted
$Y(m)$ in \cite[Section 5]{MPtricirc}; they are skeletons of maps on the torus with faces of length $6$ (see \cite[Section 5]{MPtricirc} for more details). The graphs $\Gamma_{23}(m;r,1)$ are also tricirculants, denoted $X(m)$ in \cite[Section 4]{MPtricirc}. The graphs $Y(m)$ and $X(m)$ are bicirculants if and only if $m$ is not divisible by $3$.
\end{remark}




In Section \ref{sec:cgc}, we introduce the notion of a cyclic generalised voltage graph $(\Delta,\lambda,\iota,\zeta)$, where $\lambda$, $\iota$ and $\zeta$ are integer-valued functions. 
Cyclic generalised voltage graphs are used both to define as well as 
analyse the connected, simple, cubic graphs admitting a cyclic group of automorphisms with at most three vertex-orbits. A complete classification of these graphs is stated in Section~\ref{sec:classification}.

In Section \ref{sec:gencov}, we recall the definition of a generalised voltage graph, introduced in \cite{MPgc}, and then in Section~\ref{sec:cyccov} we show that the cyclic generalised voltage graphs defined in Section~\ref{sec:cyccov}, are indeed a special case of the these. We also provide translations of some of the results about generalised voltage graph proved in \cite{MPgc} to the much simpler language of cyclic generalised voltage graph.

Using the results proved in Section~\ref{sec:cyccov}, we characterise those cyclic generalised voltage graphs that yield connected, cubic, simple covers in Section \ref{sec:ccv}. We call this class of cyclic generalised voltage graphs ccv-graphs. Using this concept we provide the proof of the classification result states in Section~\ref{sec:classification}.

The last section of this paper is devoted to determining which of the $25$ aforementioned classes admit infinitely many vertex-transitive covers and thus proving Theorem \ref{the:vt}.

\section{Concerning graphs}

Even though we are mainly interested in simple graphs, the treatment of quotients of graphs requires
a more general definition of a graph, allowing multiple edges, loops and semi-edges. This definition
has now become a standard when covers of graphs are considered; see for example \cite{lift,elabcov}.

A {\em graph} is a quadruple $(V,D,\beg,\inv)$ where $V$ is a non-empty set of {\em vertices}, $D$ is a set of {\em darts} (also known as {\em arcs} in the literature), $\beg\colon D \to V$ is a function assigning to each dart $x$ its {\em initial vertex} $\beg x$, and
$\inv\colon D \to D$ is a function assigning to each dart $x$ its {\em inverse dart} $\inv x$ (also denoted $x^{-1}$ when there is no danger for  ambiguity) satisfying 
$\inv\inv x =x$ 
 for every dart $x$.
 If $\Gamma = (V,D,\beg,\inv)$ is a graph, then we let
$\V(\Gamma):=V$, $\D(\Gamma) := D$, $\beg_\Gamma:=\beg$ and $\inv_\Gamma:=\inv$. 

 For $x\in D$, we call the vertex $\beg x^{-1}$ 
the {\em end} of $x$ and denote it by $\term x$.
 Two darts $x$ and $y$ are {\em parallel} if $\beg x = \beg y$ and $\term x = \term y$. 
An {\em edge} of a graph $\Gamma$ is a pair $\{x,x^{-1}\}$ where $x$ is a dart of $\Gamma$.
The vertices $\beg x$ and $\beg x^{-1}$ are then called the {\em endvertices} of the edge. 
An edge $\{x,x^{-1}\}$ is a {\em semi-edge} if $x=x^{-1}$ and is a {\em loop} if $x \neq x^{-1}$ but $\beg x = \term x$.

Two distinct edges $e$ and $e'$ are said to be {\em parallel} if there is a dart in $e$ that is parallel to a dart in $e'$, or equivalently, if $e$ and $e'$ have the same endvertices.
 A graph that has no semi-edges, no loops and no pairs of distinct parallel edges
is {\em simple}. Note that a simple graph is uniquely determined by its vertex-set and its edge-set and the usual terminology of simple graphs applies. In particular, a dart in a simple graph is usually called an {\em arc}.

The {\em neighbourhood}  of a vertex $v\in \V(\Gamma)$ is defined as the set $\Gamma(v):=\{x \in \D(\Gamma) : \beg x = v\}$ and the cardinality $\val_\Gamma(v):=|\Gamma(v)|$ is called the {\em valence} of $v$; if $x\in \Gamma(v)$, we also say that $x$ emanates from $v$.
 
If $\Gamma:=(V,D,\beg,\inv)$  is a graph, then we say $\Gamma':=(V',D',\beg',\inv')$ is a {\em subgraph} of $\Gamma$ (and we write $\Gamma' \leq \Gamma$) if $V' \subseteq V$, $D' \subseteq D$ and the functions $\beg'$ and $\inv'$ are the respective restrictions of $\beg$ and $\inv$ to $D'$. If additionally $V' = V$ then we say $\Gamma'$ is a {\em spanning subgraph} of $\Gamma$. 

A {\em walk}
is a sequence $(x_1,x_2,\ldots,x_n)$ where $x_i\in \D(\Gamma)$ for all $i\in \{1,\ldots,n\}$, and $\term x_i = \beg x_{i+1}$ for all $i \in \{1,\ldots,n-1\}$. 
If $\beg x_1 = u$ and $\term x_n =v$, then we say $u$ is the {\em initial vertex}, $v$ is the {\em final vertex} and the walk is called a $uv$-walk. 
The {\em inverse} of a walk $W=(x_1,x_2,\ldots,x_n)$ is the walk $W^{-1}=(x_n^{-1},x_{n-1}^{-1},\ldots,x_1^{-1})$. If $W_1 = (x_1, \ldots, x_n)$ and $W_2=(y_1, \ldots, y_m)$ are two walks such that
the final vertex of $W_1$ is equal to the initial vertex of $W_2$,
then we denote by $W_1W_2 = (x_1, \ldots, x_n,y_1, \ldots, y_m)$  the {\em concatenation} of $W_1$ and $W_2$.

A walk $W$ is {\em reduced} if it contains no two consecutive darts that are inverse to each other. 
A walk $(x_1,x_2,\ldots,x_n)$ is a {\em path} if it is reduced and $\beg x_i \neq \beg x_j$ for all $0 \leq i<j \leq n$; a {\em closed walk} if $\beg x_1 = \term x_n$, and a {\em cycle} if it is a closed path. Note that a dart underlying a semi-edge or a loop forms a cycle of length $1$.
A graph is connected if for any two vertices $u$ and $v$, there exists a $uv$-walk. A connected graph that contains no cycles is a {\em tree}. Note that a tree is always a simple graph.

A {\em morphism} of graphs $\varphi \colon \Gamma \to \Delta$
 is a function $\varphi: \V(\Gamma) \cup \D(\Gamma) \to \V(\Delta) \cup \D(\Delta) $ that maps vertices to vertices, darts to darts and such that 
 $\varphi( \inv _\Gamma x) = \inv_\Delta \varphi(x)$  and
$\varphi(\beg_\Gamma x) = \beg_\Delta \varphi(x)$.
Due to the latter condition, a morphism of graphs without isolated vertices (vertices of degree $0$)
is uniquely determined by its restriction to the set of darts; we will often exploit this fact and
define a morphism on darts only.

A surjective morphism  is called an {\em epimorphism} and a bijective morphism is called an {\em isomorphism}. An isomorphism of a graph onto itself is called an {\em automorphism}.

For a group $G$ acting on a set $\Omega$ we let $\Omega/G$ denote the
set of orbits of this action; that is $\Omega/G = \{x^G : x \in \Omega\}$.
A subset of $\Omega$ which contains precisely one element from each orbit in $\Omega/G$ is called a {\em transversal} of $\Omega/G$.

Let $\Gamma=(V,D,\beg,\Inv)$ be a graph and let $G \leq \Aut(\Gamma)$. 
Let $\beg' \colon D/G \to V/G$ and $\inv'\colon D/G \to D/G$ be mappings defined by
$\beg'(d^{G})= (\beg d)^{G}$ and $\Inv'd^G= (\Inv d)^G$. Then we call the
graph $(V/G,D/G,\beg',\Inv')$  the {\em $G$-quotient} of $\Gamma$ and denote it by $\Gamma/G$. The mapping $f:\V(\Gamma) \cup \D(\Gamma) \to \V(\Gamma/G) \cup \D(\Gamma/G)$ given by $f(x) = x^G$ is a graph epimorphism and is called the {\em quotient map} associated with $\Gamma$ and $G$.

\section{Cyclic generalised covers}
\label{sec:cgc}

We now introduce the notions of a {\em cyclic generalised voltage graph} and {\em generalised cyclic cover} which are used to define the graphs appearing in Theorem~\ref{the:main}.
We start with formal definitions. Later in the section we explain how to depict the cyclic generalised voltage graphs in the most economical manner, and
 provide an illuminating example at the end of Section~\ref{sec:cgc}.

\subsection{Formal definitions and basic results}

\begin{definition}
\label{def:cgvg}
A pair $(\Delta,\lambda)$ is a {\em dart labelled graph} provided that
 $\Delta$ is a finite connected graph and $\lambda\colon \D(\Delta) \to \NN$ is a function.
If for a dart labelled graph $(\Delta,\lambda)$ there exists a function $\iota\colon \V(\Delta) \to \NN$
such that
\begin{eqnarray}
\label{eq:ratio}\lambda(x)\iota(\beg x) &=& \lambda(x^{-1})\iota(\term x),
\end{eqnarray}
then we say that $(\Delta,\lambda)$ is {\em extendable}. If, in addition,
$\zeta\colon \D(\Delta) \to \ZZ$
is a function such that
\begin{eqnarray}
\label{eq:invvolt}\zeta(x^{-1}) &\equiv& -\zeta(x) \quad (\mod \lambda(x)\iota(\beg x))
\end{eqnarray}
for every dart $x\in \D(\Delta)$, then we say that
the quadruple $(\Delta,\lambda,\iota,\zeta)$ is a {\em cyclic generalised voltage graph}.
\end{definition}


%
\begin{figure}[h!]
\label{fig:examples}
\centering
\includegraphics[width=0.7\textwidth]{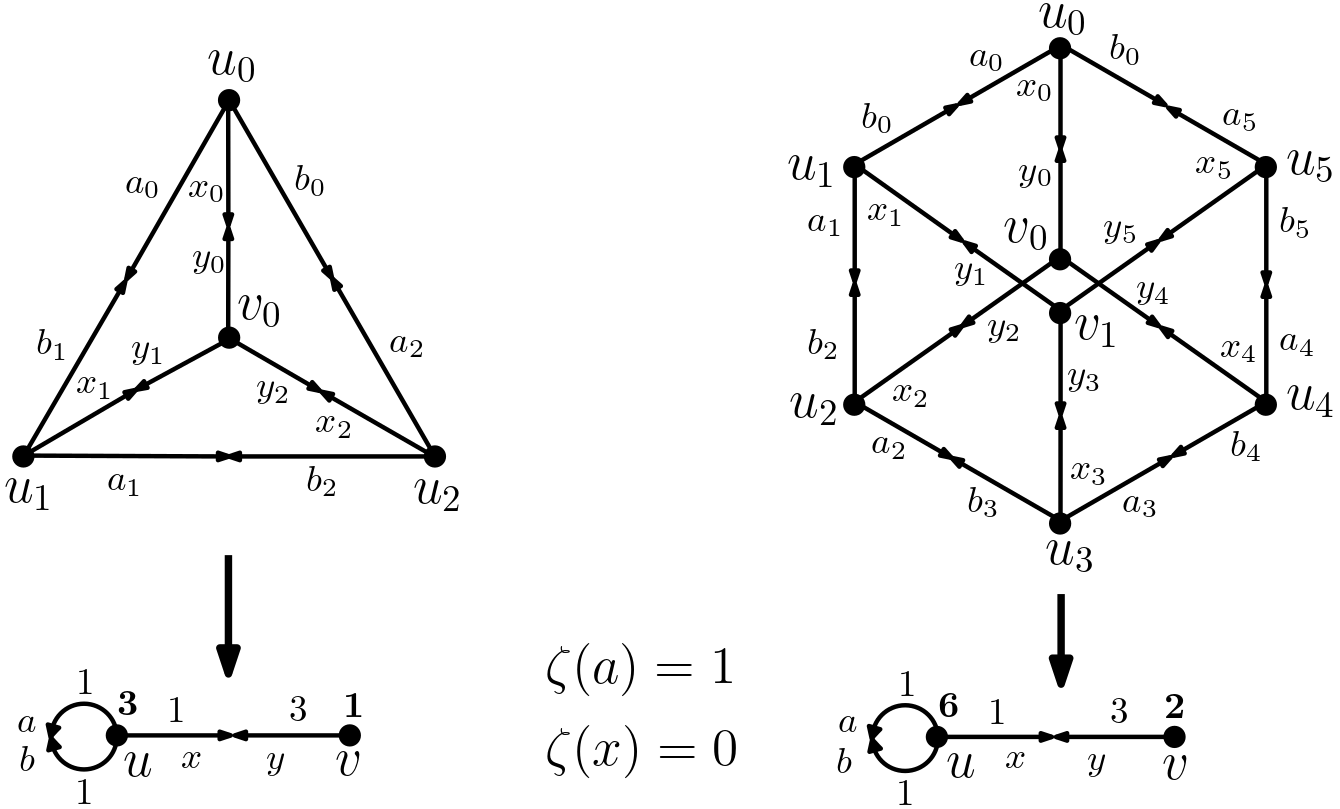}
\caption{$K_4$ and the cube graph $Q_3$ as cyclic covers of a graph with two vertices.}
\end{figure}
%


\begin{definition}
\label{def:cyc}
Let $(\Delta,\lambda,\iota,\zeta)$ be a cyclic generalised voltage graph and let $\Gamma$ be the graph defined by:
\begin{itemize}
\item $\V(\Gamma) = \{v_i \mid v \in \V(\Delta), i \in \ZZ_{\iota(v)}\}$;
\item $\D(\Gamma) = \{x_i \mid x \in \D(\Delta), i \in \ZZ_{\lambda(x)\iota(\beg x)}  \}$;
\item $\beg_{\Gamma}(x_i)= (\beg_{\Delta} x)_i$;
\item $\inv_{\Gamma}(x_i)= (\inv_{\Delta} x)_{i+\zeta(x)}$.
\end{itemize}
Then $\Gamma$ is called the {\em cyclic generalised cover} arising from $(\Delta,\lambda,\iota,\zeta)$ and is denoted by $\Cyc(\Delta,\lambda,\iota,\zeta)$.
\end{definition}

For a given vertex $v \in \V(\Delta)$, the set $\{v_i \mid i \in \ZZ_{\iota(v)}\}$ is called the fibre above $v$ and is denoted by $\fib(v)$. Similarly, for a dart $x \in \D(\Delta)$, we call $\{x_i \mid  i \in \ZZ_{\lambda(x)\iota(\beg x)}  \}$ the fibre above $x$, and denote it $\fib(x)$. The function $\pi: \V(\Gamma) \cup \D(\Gamma) \to \V(\Delta) \cup \D(\Delta)$ mapping each $x_i \in \fib(x)$ to $x$, where $x \in \V(\Gamma) \cup \D(\Gamma)$, is called the {\em generalised covering projection}.

 Of course one should verify that $\Gamma:=\Cyc(\Delta,\lambda,\iota,\zeta)$ is a well-defined graph. That is, that the function $\inv_{\Gamma}$ is well defined and $\inv_{\Gamma}\inv_{\Gamma}x=x$ for all darts $x$. Observe that condition (\ref{eq:ratio}) guarantees that the fibres $\fib(x)$ and $\fib(x^{-1})$ of two mutually inverse darts of $\Delta$ have the same size. Furthermore, it follows from condition (\ref{eq:invvolt}) that $\inv_{\Gamma}\inv_{\Gamma}x=x$. Hence  $\inv_{\Gamma}$ is a well-defined function and $\Gamma$ is a graph. 

Let $(\Delta,\lambda,\iota,\zeta)$ be a cyclic generalised voltage graph, let $x \in \D(\Delta)$ and $u = \beg x$. Consider the cyclic generalised cover $\Gamma:=\Cyc(\Delta,\lambda,\iota,\zeta)$. From Definition \ref{def:cyc} we see that a dart $x_j \in \fib(x)$ emanates from a vertex $u_i \in \fib(u)$ if and only if $j \equiv i$ ($\mod \iota(u)$). Clearly, there are exactly $\lambda(x)$ elements of $\ZZ_{\lambda(x)\iota(u)}$ congruent to $i$ modulo $\iota(u)$. Then, since $|\fib(x)|=\lambda(x)\iota(u)$, we see that $\lambda(x)$ is the number of darts in $\fib(x)$ that are incident to any given vertex in $\fib(u)$. In particular this implies that the valency of a vertex $u_i \in \fib(u)$ is given by the formula 
\begin{align}
\label{eq:degree}
\val_\Gamma(u_i) = \sum\limits_{x \in \Delta(u)} \lambda(x).
\end{align}

Lemma \ref{lem:adjacency} below, which follows almost directly from Definition \ref{def:cyc}, gives in rather simple terms the adjacency rule for the cover of a cyclic generalised voltage graph, and will prove to be quite useful in the following pages.  

\begin{lemma}
\label{lem:adjacency}
Let $(\Delta,\lambda,\iota,\zeta)$ be a cyclic generalised voltage graph let $u,v \in \V(\Delta)$ and let $\Gamma = \Cyc(\Delta,\lambda,\iota,\zeta)$. Two vertices $u_i \in \fib(u)$ and $v_j \in \fib(v)$ are adjacent if and only if 
\begin{align}
j - i \equiv  \zeta(x) \quad (\mod \gcd(\iota(u), \iota(v)))
\end{align}
for some  $x \in \D(\Delta)$ such that $\beg_\Delta x = u$ and $\term_\Delta x = v$.
\end{lemma}

\begin{proof}
Observe that $u_i$ is adjacent to $v_j$ if and only if there exist a dart $x \in \D(\Delta)$ with $\beg_\Delta x =u$ and $\term_\Delta x = v$, and a dart $x_{\ell} \in \fib(x)$ such that $\beg_\Gamma x_{\ell} = u_i$ and $\term_\Gamma x_{\ell} = v_j$. By Definition~\ref{def:cyc}, we see that
\begin{align*}
\beg_\Gamma x_{\ell}  =  (\beg_\Delta x)_\ell = u_\ell\,;
\end{align*}
\begin{align*}
\term_\Gamma x_{\ell}  =  \beg_\Gamma (\Inv_\Gamma x_{\ell}) =
\beg_\Gamma ((\Inv_\Delta x)_{\ell+\zeta(x)} ) = (\term_\Delta  x)_{\ell+\zeta(x)} = v_{\ell+\zeta(x)}.
\end{align*}
Recall that the indices at $u$ and $v$ should be considered as elements of $\ZZ_{\iota(u)}$ and $\ZZ_{\iota(v)}$, respectively. Hence $u_i$ is adjacent to $v_j$ if and only if
$i \equiv \ell\> (\mod \iota(u))$ and $j \equiv \ell + \zeta(x)\> (\mod \iota(v))$
for some integer $\ell$. The latter is equivalent to the existence of integers $X,Y$ and $\ell$
such that $\iota(u)X +\ell = i$ and $\iota(v)Y +\ell = j - \zeta(x)$. By subtracting these two equalities,
we see that in this case the diophantine equation $\iota(u)X - \iota(v)Y = i-j+\zeta(x)$ has a solution,
which is equivalent to requiring that $j - i \equiv  \zeta(x) \quad (\mod \gcd(\iota(u), \iota(v)))$.
Conversely, if this diophantine equation has a solution, then by letting $\ell := i -\iota(u)X$, which equals
$j - \iota(v)Y - \zeta(x)$, we see that $i \equiv \ell\> (\mod \iota(u))$ and $j \equiv \ell + \zeta(x)\> (\mod \iota(v))$, and thus $u_i$ is adjacent to $v_j$.
\end{proof}

\begin{remark}
\label{rem:reduced} 
Suppose $(\Delta,\lambda,\iota,\zeta)$ is a cyclic generalised voltage graph and let $\zeta':\D(\Delta) \to \ZZ$ be the function obtained from $\zeta$ by reducing every $\zeta(x)$ modulo $\gcd(\iota(\beg x),\iota(\term x))$. That is, for each dart $x \in \D(\Delta)$, $\zeta'(x)$ is the unique integer such that $\zeta'(x) < d_x$ and $\zeta'(x) \equiv \zeta(x)$ $(\mod d_x)$ where $d_x = \gcd(\iota(\beg x),\iota(\term x))$. Then, from Lemma~\ref{lem:adjacency} we have $\Cyc(\Delta,\lambda,\iota,\zeta)=\Cyc(\Delta,\lambda,\iota,\zeta')$.
\end{remark}

This construction is illustrated in Figure \ref{fig:examples}. Consider the cyclic generalised voltage graph $(\Delta,\lambda,\iota,\zeta)$ with darts $a$, $b$, $x$ and $y$, and vertices $u$ and $v$, shown in the bottom left of the figure. Its cyclic generalised cover, $\Gamma$, is shown immediately above it. Next to each vertex, in bold lettering, is its image under $\iota$. This is, $\iota(u) = 3$ and $\iota (v) = 1$. Then, there are $3$ distinct vertices in the fibre of $u$ (labelled $u_i$ with $i \in \ZZ_3$) and a single vertex in the fibre of $v$. Next to each dart is a number indicating its label $\lambda$. Since $\lambda(y)=3$ and $\beg(y)=v$, there are $3$ darts in the fibre of $y$ beginning at each vertex in the fibre of $v$ (which consists solely of $v_0$). Similarly, for each vertex $u_i$ in the fibre of $u$ there is a single dart in each of the fibres of $x$, $a$ and $b$, that begins at $u_i$.

The voltages are given by $\zeta(a)=1$, $\zeta(b) = 2$ and $\zeta(x) = \zeta(y) = 0$. Consider the dart $a$ and see that, by Lemma \ref{lem:adjacency} (and since $\zeta(a)=1$, $\beg(a)=u$ and $\term(a)=u$) each vertex $u_i$ is adjacent to all vertices $u_j$ such that $j \equiv i + \zeta(a) = i + 1$ modulo $\iota(u) = 3$. This is, each $u_i$ is adjacent to $u_{i+1}$. By the same token, but considering the dart $b$, we see that each $u_i$ is adjacent to $u_{i+2}$. Finally, Lemma \ref{lem:adjacency} also tells us that $v_0$ is adjacent to all $u_0$, $u_1$ and $u_2$.  

\subsection{Extendability of labelled graphs}

We now turn our attention to the question of which labelled graphs $(\Delta,\lambda)$ are
extendable (in the sense of Definition~\ref{def:cgvg}) and which functions $\iota$ then satisfy condition (\ref{eq:ratio}). 
%
 This will then allow us to describe a cyclic generalised voltage graph
in a more economical manner.

For a function $\lambda \colon \D(\Delta) \to \NN$ and a walk $W=(x_1,\ldots,x_n)$ of $\Delta$, let
\begin{align}
\label{eq:rho}
 \rho_\lambda(W) := \prod\limits_{i=1}^n \frac{\lambda(x_i)}{\lambda(x_i^{-1})}.
\end{align}

Then clearly
\begin{align}
\label{eq:rhoW1W2}
\rho_\lambda(W^{-1}) = \rho_\lambda(W)^{-1} \> \hbox{ and } \>
 \rho_\lambda(W_1W_2) = \rho_\lambda(W_1)\rho_\lambda(W_2)
\end{align}
for every two walks for which the concatenation $W_1W_2$ is defined.

Now observe that if 
$\iota\colon \V(\Delta) \to \NN$ is a function which together with $\lambda$ satisfies
the condition (\ref{eq:ratio}), then a repeated application of (\ref{eq:ratio}) to
the darts $x_i$ of the walk $W$ implies that
\begin{align}
\label{eq:li}
\iota(\term x_n)
= \iota(\beg(x_1)) \rho_\lambda(W).
\end{align}

In particular, $\rho_\lambda(W) = 1$ for every closed walk $W$.
This observations has the following useful consequence.

\begin{lemma}
\label{lem:consistent}
A  labelled graph $(\Delta,\lambda)$ is extendable if and only if
$\rho_\lambda(W) = 1$ for every closed walk $W$ of $\Delta$.

Let $u\in \V(\Delta)$, $m\in \NN$ and suppose that $(\Delta,\lambda)$ is extendable.
Then a function $\iota \colon \V(\Delta) \to \NN$ satisfying (\ref{eq:ratio})
and such that $\iota(u) = m$ exists
if and only if $m$ is divisible by $\lcm\{\alpha(v) : v\in \V(\Delta)\}$, where 
$\alpha(v)$ is the smallest integer for which $\alpha(v)\rho_\lambda(W_v) \in \NN$
for one (and thus every) $uv$-walk $W_v$. Such a function $\iota$ is then unique.
\end{lemma}

\begin{proof}
Fix a vertex $u$ of $\Delta$. 
Suppose first that $(\Delta,\lambda)$ is extendable and that $\iota\colon \V(\Delta) \to \NN$ is a function
that together with $\lambda$  satisfies (\ref{eq:ratio}). In view of (\ref{eq:li}), it follows that
$\rho_\lambda(W) = 1$ for every closed walk $W$. Moreover, if $\iota(u) = m$, then
$\iota(v) = m \rho_\lambda(W_v)$ where $W_v$ is an arbitrary $uv$-walk. Since $\iota(v)$ is an integer,
this implies that $m$ is divisible by $\alpha(v)$. Since this is true for every vertex $v\in \V(\Delta)$,
we see that $m$ is divisible by $\lcm\{\alpha(v) : v\in \V(\Delta)\}$, as claimed.

Suppose now that $\rho_\lambda(W) = 1$ for every closed walk $W$.
Let $v\in \V(\Delta)$ and
suppose that $W_1$ and $W_2$ are two $uv$-walks. Then $W_1W_2^{-1}$ is a closed walk
and hence $\rho_\lambda(W_1W_2^{-1}) = 1$. But then, in view of (\ref{eq:rhoW1W2}),
$\rho_\lambda(W_1) =\rho_\lambda(W_2)$.

Now choose a $uv$-walk $W_v$, let $\alpha(v)$ be the smallest integer such that 
$\alpha(v)\rho_\lambda(W_v) \in \NN$ 
and let
\begin{eqnarray}
\label{eq:ext1}\iota(u) &:=& \lcm\{\alpha(v) \mid v \in \V\}; \\ 
\label{eq:ext2}\iota(v) &:= &\iota(u) \rho_\lambda(W_v)\> \hbox { if } v\not = u.
\end{eqnarray}
In view of the the above argument, we see that this defines  $\alpha(v)$ and $\iota(v)$ independently of the choice of the $uv$-walk $W_v$.

To show that  $\iota$ together with $\lambda$ satisfies (\ref{eq:ratio}),
 let $x \in \D(\Delta)$ and let $X$ be the walk of length $1$ consisting only of the dart $x$. 
 Then 
 $
 \iota(\term x)= \iota(u)\rho_\lambda(W_{\term x})=$
$\iota(u)\rho_\lambda(W_{\beg x})\rho_\lambda(X)=$
$\iota(\beg x)\rho_\lambda(X)=\iota(\beg x)\frac{\lambda(x)}{\lambda(x^{-1})}$
and (\ref{eq:ratio}) thus holds.
\end{proof}

\begin{remark}
Note that the condition $\rho_\lambda(W) = 1$ in Lemma~\ref{lem:consistent} only needs to be checked for all cycles $W$. To see that, assume that $\rho_\lambda(W) = 1$ for every cycle $W$ of $\Delta$ and let $W=(x_1,\ldots,x_n)$ be a closed walk of shortest length such that $\rho_\lambda(W) \not= 1$. 
Then $W$ is not a cycle. However, $W$ is a reduced walk since otherwise, by removing a pair of consecutive mutually inverse darts, we obtain a shorter walk $W'$ with $\rho_\lambda(W') = \rho_\lambda(W) \not = 1$, a contradiction. Hence
 there are $i,j$, $1\le i<j\le n$ with $\beg x_i = \beg x_j$.
Let $W_1 = (x_1, \ldots, x_{i-1}, x_j, \ldots,x_n)$ and 
$W_2 = (x_i, \ldots, x_{j-1})$. By the minimality of $W$, we see that $\rho_\lambda(W_i) = 1$ for all $i\in\{1,2\}$. On the other hand, clearly $\rho_\lambda(W)=\rho_\lambda(W_1)\rho_\lambda(W_2)$, a contradiction.

Moreover, an easy argument (which we leave out) shows that
is suffices to check the condition $\rho_\lambda(W) = 1$ on
any complete set of fundamental cycles (recall that
a complete set of {\em fundamental cycles relative to a fixed spanning tree} $T$ is the set of all cycles that contain exactly one cotree dart).
\end{remark}

\begin{remark}
\label{rem:index}
The function $\iota$ that we define in the proof of Lemma \ref{lem:consistent} is such that $\iota(u)$ is the least positive integer $m$ such that $m\rho_\lambda(W_v)$ is an integer for all $v \in \V(\Delta)$.
It follows that $\gcd\{\iota(v) \mid v \in \V(\Delta)\}=1$. This has two straightforward but important consequences. First, if $\iota':\D(\Delta) \to \NN$ is any function that together with $\lambda$ satisfies (\ref{eq:ratio}), 
then $\iota' = c \cdot \iota$ for some positive integer $c$. Second, from Theorem~\ref{theo:concyc}, we see that $\Cyc(\Delta,\lambda,\iota,\zeta)$, where $\zeta$ is defined to be $0$ on every dart, is a connected graph. 
\end{remark}


\subsection{Depicting cyclic generalised voltage graphs}

We will try to be as economical as possible when drawing a cyclic generalised voltage graph $(\Delta,\lambda,\iota,\zeta)$. We will draw every edge simply as a line segment connecting its endpoints, every loop as a closed curve beginning and ending at its unique endpoint, and every semi-edge as a pendant line segment. We will write the label $\lambda(x)$ near to each dart that belongs to an edge, so an edge will have two labels, one next to each of its endpoints. We will omit the label on a dart $x$ whenever $\lambda(x)=\lambda(x^{-1})=1$. Moreover, since the index function $\iota$ is completely determined by its image on one vertex, it suffices to specify the image of $\iota$ on a distinguished vertex $v$. We will often use the letter $m$ to denote $\iota(v)$. To indicate voltages, we will draw an arrowhead in the middle of every loop or edge $\{x,x^{-1}\}$. We will write $\zeta(x)$ next to the arrowhead if it is oriented from $\beg x$ to $\term x$ (otherwise, we will write $\zeta(x^{-1})$). We will omit the voltage on edges and loops having trivial voltage. A semi-edge $x$ with no specified voltage in the picture will be assumed to have voltage $\iota(\beg x )/2$.


\begin{figure}[h]
\centering
\includegraphics[width=0.4\textwidth]{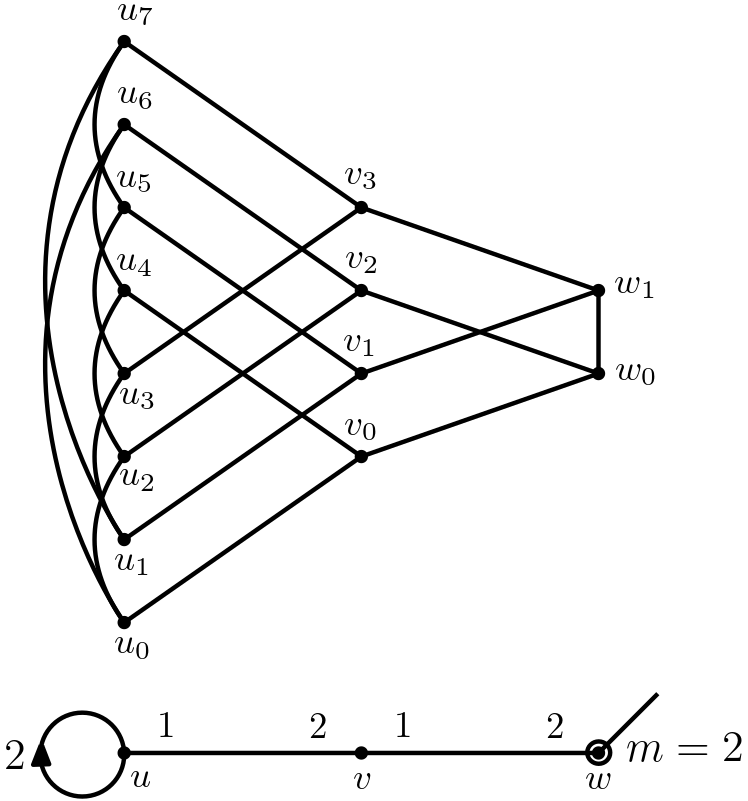}
\caption{A diagram and its covering graph}
\label{fig:firstexample}
\end{figure}

We now provide an example that illustrates the above described economical way of depicting the cyclic generalised voltage graphs and describe the
cyclic generalised cover arising from it in as much an intuitive way as possible.

Consider the cyclic generalised voltage graph $(\Delta,\lambda,\iota,\zeta)$ in the bottom of Figure~\ref{fig:firstexample}, consisting of
three vertices $u$, $v$ and $w$,  and having one {\em loop} attached to the vertex $u$, two ordinary edges linking
$u$ with $v$ and $v$ with $w$, respectively, and
a {\em semi-edge} attached to $w$. We will
explain how it gives rise
to its {\em cyclic generalised cover}.

The index function $\iota$ is given as follows. First, note that $w$ is the distinguished vertex and has a positive integer $m$ (in this case $m=2$) assigned to it. This indicates that $\iota(w)=m$ or, simply put, that the vertex $w$ gives rise to $m$ distinct vertices in $\Gamma$ (the fibre above $w$). The images of other vertices under $\iota$ (and thus also the size of the fibres above those vertices) are determined by $\iota(w)$. Recall that for all darts $x$ we have $\lambda(x)\iota(\beg x) = \lambda(x^{-1})\iota(\term x)$. By letting $x$ be the dart beginning at $w$ and ending at $v$ we have $\lambda(x)\iota(w) = \lambda(x^{-1})\iota(v)$, and so $\iota(v) = 2 \cdot \iota(w)=4$. Similarly, $\iota(u) = 2 \iota(v) = 8$. 

Adjacency rules are given by Lemma \ref{lem:adjacency}. Recall that an edge lacking an arrowhead indicates that both darts comprising it have trivial voltage. For instance all darts in the edge between $w$ and $v$, as well as those in the edge between $v$ and $u$, have trivial voltage. We see that a vertex $w_i \in \fib(w)$ is adjacent to all vertices $v_j \in \fib(v)$ such that $j \equiv i$ $(\mod 2)$. This is each $w_i$ is adjacent to $v_i$ and $v_{i+2}$. Similarly each $v_i$ is adjacent to $u_i$ and $u_{i+4}$. Now, the loop at $u$ has an arrowhead with the number $2$ written next to it. This is, one dart in the loop at $u$ has voltage $2$ while the other has voltage $-2$ (or equivalently, $6$, since the indices of the vertices in $\fib(u)$ are taken modulo $\iota(u)=8$). Then every $u_i \in \fib(u)$ is adjacent to $u_{i+2}$ and $u_{i-2}$. Finally, the semi-edge at $w$ has no specified voltage in the picture. This means that it has voltage $\iota(w)/2 = 1$. Then every vertex $w_i \in \fib(w)$ is adjacent to $w_{i+1}$.

\section{Classification theorem}
\label{sec:classification}

We now use the language of cyclic generalised covers to present a classification of  all connected simple cubic graphs admitting an automorphism with at most three vertex-orbits. 

\begin{definition}
\label{def:ccv}
Let $(\Delta,\lambda,\iota,\zeta)$ be a cyclic generalised voltage graph. We say that $(\Delta,\lambda,\iota,\zeta)$ is a cyclic cubic voltage graph (or simply a $\ccv$-graph) if $\Cyc(\Delta,\lambda,\iota,\zeta)$ is a simple connected cubic (that is, $3$-regular) graph.
\end{definition}

Let us now turn our attention to Figure \ref{fig:quotients}. The $25$ graphs depicted here, which we denote by $\Delta_i$ for $i \in \{1,\ldots,25\}$, are all the dart-labelled graphs on at most $3$ vertices that can be extended to a ccv-graph. The proof that this family of $25$ dart-labelled graphs is indeed complete is postponed until Section~\ref{sec:ccv}, where it was obtained by checking all the graphs on at most $3$ vertices and maximal valence at most $3$,
and then for each of these graphs finding all the labelings $\lambda$
satisfying the conditions given in Proposition~\ref{prop:consistentccv} (note that this is a finite problem for each graph).
For each dart-labelled graph $\Delta_i$ we have chosen a distinguished vertex shown as a white vertex with a black circle around it and the letter $m$ next to it. A voltage assignment, which may assign  values $r$ or $s$ to some darts, has also been specified in the drawing. 

\begin{figure}[h!]
\includegraphics[width=0.75\textwidth]{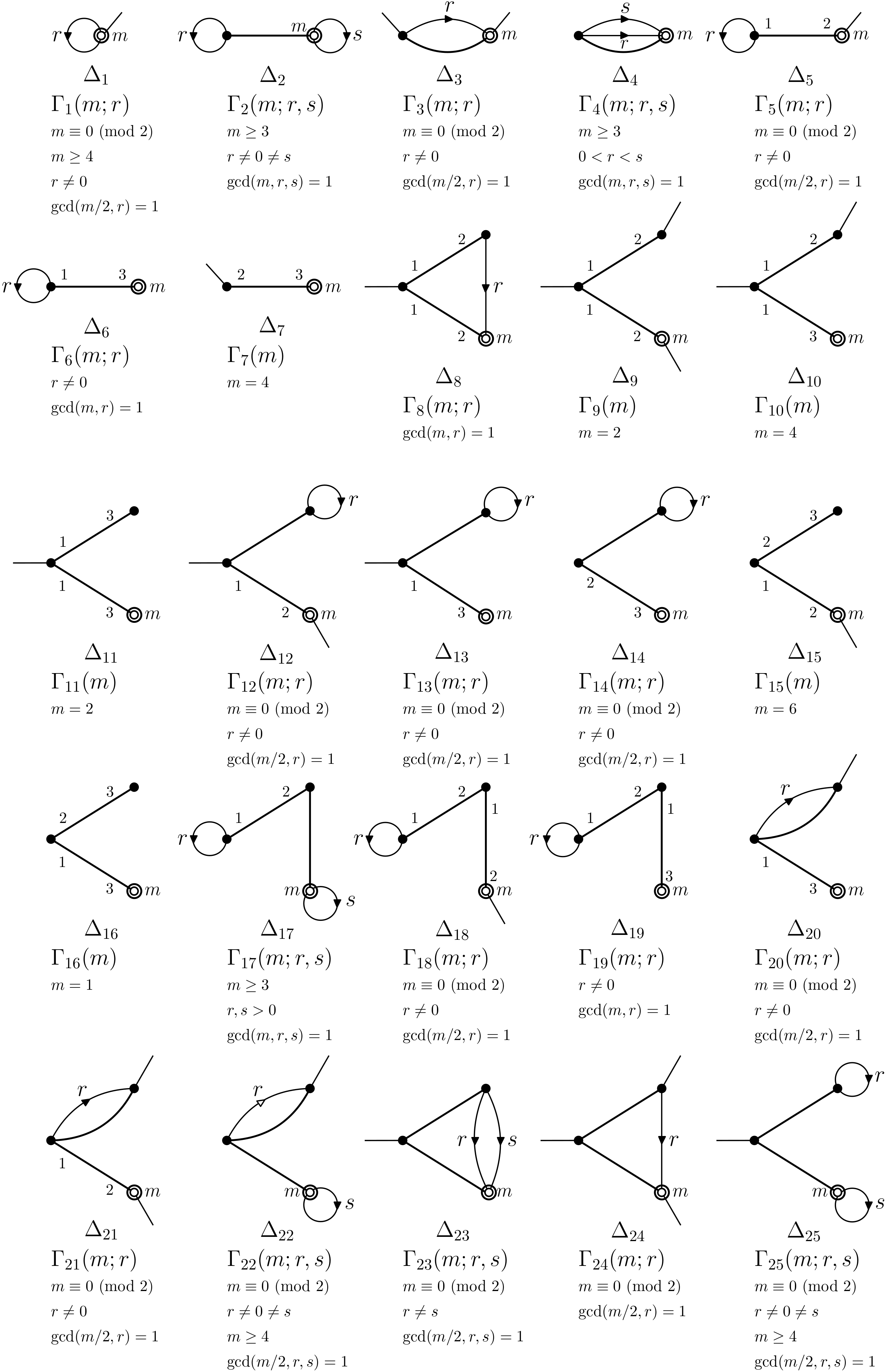}
\caption{The $25$ dart-labelled graphs on at most three vertices that can be extended to a ccv-graph.}
\label{fig:quotients}
\end{figure}

Let $i \in \{1,\ldots,25\}$ and let $m,r,s > 0$. We define $\Gamma_i(m;r,s)$ (alternatively $\Gamma_i(m;r)$ or $\Gamma_i(m)$) as the cyclic generalised cover of the cyclic generalised voltage graph $(\Delta,\lambda,\iota,\zeta)$ where $(\Delta,\lambda)=\Delta_i$, $\iota$ is the index function defined by $\iota(v_0)=m$, where $v_0$ is the distinguished vertex of $\Delta_i$, and $\zeta$ is the voltage assignment shown in the picture, where we take the voltage of a semi-edge $x$ to be $\iota(\beg x)/2$. We will also assume that the integers $m$, $r$ and $s$ satisfy the conditions listed underneath $\Delta_i$ in Figure \ref{fig:quotients}. As we will see in Section \ref{sec:ccv}, these conditions are sufficent and necessary for $(\Delta,\lambda,\iota,\zeta)$ to be a $ccv$-graph. Hence, the graph $\Gamma_i(m;r,s)$ (or $\Gamma_i(m;r)$ or $\Gamma_i(m)$, accordingly) is a connected, cubic, simple graph. We will say that a graph $\Gamma$ is a ccv-cover of $\Delta_i$ for some $i \in \{1,\ldots,25\}$ if $\Gamma$ is isomorphic to $\Gamma_i(m;r,s)$ (or $\Gamma_i(m;r)$ or $\Gamma_i(m)$, accordingly).

Consider for instance the labelled graph $\Delta_6$. The covering graph $\Gamma_6(1;1)$ is the cyclic generalised cover of the ccv-graph obtained by assigning voltage $r=1$ to (one dart underlying) the loop and letting $\iota(v_0)=1$. Then, if $u$ is the remaining vertex, we see that $\iota(u) = 3$. Observe that the ccv-graph thus obtained is precisely the generalised cyclic voltage graph shown in the bottom left of Figure \ref{fig:examples}. Then $\Gamma_6(1;1)$ is isomorphic to $K_4$, as shown in the figure. Similarly, $\Gamma_6(2;1)$ is isomorphic to the cube $Q_3$ (see Figure \ref{fig:examples}). The graph constructed in the example in the end of the preceding chapter and shown in Figure \ref{fig:firstexample} is the graph $\Gamma_{18}(2;2)$.

The proof of the following theorem will be given in Section~\ref{ssec:proof}.

\begin{theorem}
\label{the:main}
A graph $\Gamma$ is a connected, cubic, simple graph with a cyclic group of automorphisms $G \leq \Aut(\Gamma)$ having at most $3$ orbits on vertices if and only if it is isomorphic to one of the following:
\begin{enumerate}
\item $\Gamma_7(2)$, $\Gamma_{9}(4)$, $\Gamma_{10}(2)$, $\Gamma_{11}(4)$,  $\Gamma_{15}(6)$ or $\Gamma_{16}(1)$;
\item $\Gamma_1(m;r)$ with $m$ even, $m \geq 4$, $r \neq 0$, $\gcd(\frac{m}{2},r)=1$;
\item $\Gamma_{i}(m;r)$ with $i \in \{3,12,13,14,18,20,21\}$, $m$ even, $r \neq 0$, $\gcd(\frac{m}{2},r)=1$;
\item $\Gamma_i(m;r)$ with $i \in \{6,19\}$, $r \neq 0$, $\gcd(m,r)=1$;
\item $\Gamma_8(m;r)$ with $\gcd(m,r)=1$;
\item $\Gamma_{i}(m;r,s)$ with $i \in \{2,17\}$, $m \geq 3$, $r \neq 0 \neq s$, $\gcd(m,r,s)=1$;
\item $\Gamma_4(m;r,s)$ with $m \geq 3$, $0 < r < s$, $\gcd(m,r,s)=1$;
\item $\Gamma_5(m;r,s)$ with $m$ even, $r \neq 0$, $\gcd(\frac{m}{2},r)=1$;
\item $\Gamma_{23}(m;r,s)$ with $m$ even, $r \neq s$, $\gcd(\frac{m}{2},r)=1$;
\item $\Gamma_{24}(m;r,s)$ with $m$ even, $\gcd(\frac{m}{2},r)=1$;
\item $\Gamma_{i}(m;r,s)$ with $i \in \{22,25\}$, $m$ even, $m \geq 4$, $r \neq 0 \neq s$, $\gcd(\frac{m}{2},r,s)=1$.
\end{enumerate}
\end{theorem}

\section{Generalised covers}
\label{sec:gencov}

In this and next section we prove that cyclic generalised voltage graphs
are a special case of a much more general notion of a {\em generalised voltage graph}. This fact will come useful in Section~\ref{sec:cyccov}, where
we translate several facts about generalised voltage graphs proved in \cite{MPgc} into the language of cyclic generalised voltage graphs.

For a group $G$, we let $S(G)$ denote the set of all subgroups of $G$ and for $g,h\in G$, we let $g^h:=h^{-1}gh$ be the conjugate of $g$ by $h$.
The {\em core} of a subgroup $H$ in a group $G$, denoted $\core_G(H)$ is the intersection $\cap_{x\in G}{H^x}$ of all the conjugates of $H$.
 We can now present the notions of the {\em generalised voltage graph} and of the {\em generalised cover} introduced in \cite{MPgc}.

\begin{definition}
Let $\Delta$ be a connected graph, let $G$ be a group, and 
let $\omega\colon \V(\Delta) \cup \D(\Delta) \to S(G)$ and $\zeta\colon \D(\Delta) \to G$ be two functions such that the following hold for all $x \in \D(\Delta)$:
\begin{align}
\label{eq:volt1} \omega(x) \leq  \omega(\beg_\Delta x);\\ 
\label{eq:volt2} \omega(x) = \omega(x^{-1})^{\zeta(x)};\\
\label{eq:volt3} \zeta(x^{-1})\zeta(x) \in \omega(x). 	
\end{align}

We then say that the quadruple $(\Delta,G,\omega,\zeta)$ is a {\em generalised voltage graph} and we call the functions $\omega$ and $\zeta$ a {\em weight function} and a {\em voltage assignment}, respectively. 
The generalised voltage graph $(\Delta,G,\omega,\zeta)$ is 
{\em faithful} provided that
  $\core_G\big(\bigcap_{x\in \D(\Delta)} \omega(x)\big) = 1.$
\end{definition}

\begin{definition}
\label{def:gencov}
Let $(\Delta,G,\omega,\zeta)$ be a generalised voltage graph and let $\Gamma$ be the graph defined by:
\begin{itemize}
\item $\V(\Gamma) = \{(v,\omega(v)g) \mid g \in G, v \in \V(\Delta)\}$;
\item $\D(\Gamma) = \{(x,\omega(x)g) \mid g \in G, x \in \D(\Delta)\}$;
\item $\beg_{\Gamma}(x, \omega(x)g)=(\beg_{\Delta} (x), \omega(\beg_{\Delta} x)g)$;
\item $\inv_{\Gamma}(x, \omega(x)g)=(\inv_\Delta x, \omega(\inv_\Delta x)\zeta(x) g)$.
\end{itemize}
Then $\Gamma$ is called the {\em generalised cover} arising from $(\Delta,G,\omega,\zeta)$ and is denoted by $\GC(\Delta,G,\omega,\zeta)$.
\end{definition}


The following theorem is the crucial fact of the theory of generalised covers and states that a graph $\Gamma$ admitting a group of automorphisms $G$ can be reconstructed from a voltage graph $(\Gamma/G,G,\omega,\zeta)$ (for some appropriate $\omega$ and $\zeta$). This theorem follows directly from Theorem~14 and Remark~17 of \cite{MPgc}. 

\begin{theorem}
\label{the:lift}
Let $\Gamma$ be a graph and let $G \leq \Aut(\Gamma)$. Then there exist functions $\zeta \colon \D(\Gamma/G) \to G$ and $\omega \colon \V(\Gamma/G) \cup \D(\Gamma/G) \to S(G)$, such that $(\Gamma/G,G,\omega,\zeta)$ is a faithful generalised faithful voltage graph and $\Gamma$ is isomorphic to the associated generalised covering graph $\GC(\Gamma/G,G,\omega,\zeta)$.
\end{theorem}

\section{Cyclic generalised covers as a special case}
\label{sec:cyccov}

Let $(\Delta,\ZZ_n,\omega,\zeta)$ be a faithful generalised voltage graph. Since $\ZZ_n$ is cyclic, every subgroup $\omega(x) \leq \ZZ_n$, is determined uniquely by its index in $\ZZ_n$. This allows us to define two functions $\lambda$ and $\iota$ such that $(\Delta,\lambda,\iota,\zeta)$ is a cyclic generalised voltage graph and $\GC(\Delta,\ZZ_n,\omega,\zeta) \cong \Cyc(\Delta,\lambda,\iota,\zeta)$. Informally, we can think of a cyclic generalised voltage graph simply as a particular type of a generalised voltage graph. 
 Let us make this more precise.
 

Let $\ZZ_n$ be the cyclic group of integers modulo $n$, and let $(\Delta,\ZZ_n,\omega,\zeta)$ be a generalised voltage graph. Define $\iota: \V(\Delta) \to \ZZ$ and $\lambda: \D(\Delta) \to \ZZ$ by 
\begin{align}
\label{eq:iotaC} \iota(v) & =|\ZZ_n:\omega(v)|; \\
\label{eq:lambdaC} \lambda(x) & = |\omega(\beg x):\omega(x)|.
\end{align}
Now observe that $\omega(v)$ is the subgroup generated by $\iota(v)$, where we view the integer $\iota(v)$ as an element of $\ZZ_n$. Similarly, $\omega(x)$ is generated by its index in $\ZZ_n$, which equals $|\ZZ_n:\omega(x)|=|\ZZ_n:\omega(\beg x)|\cdot|\omega(\beg x):\omega(x)| = \iota(\beg x)\lambda(x)$. Hence, the function $\omega$ can be reconstructed from the integer valued functions $\iota\colon \V(\Delta) \to \ZZ$ and $\lambda\colon \D(\Delta) \to \ZZ$ using the formulas
\begin{align}
\label{eq:omega1}\omega(v) &= \langle \iota(v) \rangle;\\
\label{eq:omega2}\omega(x) &= \langle \iota(\beg x)\lambda(x)\rangle.
\end{align}
Furthermore, if the generalised voltage graph $(\Delta,\ZZ_n,\omega,\zeta)$ is faithful, the order $n$ of $\ZZ_n$ is determined by the functions $\iota$ and $\lambda$, as shown in the lemma below.

\begin{lemma}
\label{lem:cycfaithful}
If $(\Delta,\ZZ_n,\omega,\zeta)$ is a generalised voltage graph, then it is faithful if and only if
\begin{equation}
\label{eq:n}
n = \lcm\{\iota(\beg x)\lambda(x) \mid x \in \D(\Delta)\}.
\end{equation}
\end{lemma}

\begin{proof}
Let $G=\ZZ_n$.
By definition, $(\Delta,G,\omega,\zeta)$ is faithful if and only if $\core_{G}(\cap_{x \in \D(\Delta)} \omega(x)) = 1$. Since $G$ is abelian and since $\omega(x) = \langle \iota(\beg x)\lambda(x) \rangle$, this is equivalent to the condition that the group $$ H:= \bigcap\limits_{x \in \D(\Delta)} \langle \iota(\beg x)\lambda(x) \rangle$$ is trivial. Observe that $H$ is precisely the subgroup of $\ZZ_n$ generated by  $\lcm\{\iota(\beg x)\lambda(x)  \mid x \in \D(\Delta)\}$. Finally, $|\langle \lcm\{\iota(\beg x)\lambda(x)  \mid x \in \D(\Delta)\}\rangle|= 1$ if and only if $\lcm\{\iota(\beg x)\lambda(x)  \mid x \in \D(\Delta)\} = n$. 
\end{proof}

The discussion so far allows us to define a function $\Phi$ which assigns to each faithful generalised voltage graph $(\Delta,\ZZ_n,\omega,\zeta)$ a cyclic generalised voltage graph $\Phi(\Delta,\ZZ_n,\omega,\zeta):=(\Delta,\lambda,\iota,\zeta)$ where $\iota$ and $\lambda$ are given by (\ref{eq:iotaC}) and (\ref{eq:lambdaC}). Conversely, one can define a mapping $\Psi$ that assigns to each cyclic generalised voltage graph, $(\Delta,\lambda,\iota,\zeta)$ a faithful generalised voltage graph $\Psi(\Delta,\lambda,\iota,\zeta):=(\Delta,\ZZ_n,\omega,\zeta)$, where $\omega$ and $n$ are given by (\ref{eq:omega1}), (\ref{eq:omega2}) and (\ref{eq:n}). The mappings $\Phi$ and $\Psi$ are inverse to each other.

\begin{theorem}
\label{theo:corresp}
Let $(\Delta,\ZZ_n,\omega,\zeta)$ be a generalised voltage graph and let $(\Delta,\lambda,\iota,\zeta)=\Phi(\Delta,\ZZ_n,\omega,\zeta)$. Then $\Cyc(\Delta,\lambda,\iota,\zeta) \cong \GC(\Delta,\ZZ_n,\omega,\zeta)$.
\end{theorem}

\begin{proof}
Let $\varphi\colon \D(\GC(\Delta,\ZZ_n,\omega,\zeta)) \to \D(\Cyc(\Delta,\lambda,\iota,\zeta))$ be given by
\begin{align*}
\varphi(x,\omega(x) + i) = x_i,
\end{align*}
where the coset of $\omega(x) \leq \ZZ_n$ is written in additive notation and the subindex of $x_i$ is taken modulo $\lambda(x)\iota(\beg x)$. To see that $\varphi$ is well defined suppose $\omega(x)+i = \omega(x)+j$. Then $i - j \in \omega(x)=\langle \lambda(x)\iota(\beg x) \rangle$ which implies that $i \equiv j$ $(\mod \lambda(x)\iota(\beg x))$. This is $x_i = x_j$. It is straightforward to verify that $\varphi$ is indeed an isomorphism.
\end{proof}


We devote the remainder of this section to stating some fundamental properties of cyclic generalised covers. Each proposition hereafter is a special case, where the voltage group $G$ is cyclic, of a more general proposition for a generalised voltage graph $(\Delta,G,\omega,\zeta)$ proved in \cite{MPgc}. We state these results here in the language of cyclic generalised voltage graphs. Even though each of the following propositions is simply a translation of a (special case of a) proposition stated and proved in \cite{MPgc}, we provide a sketch of a proof for Theorems \ref{theo:concyc} and \ref{prop:simple}, as in these cases the translation might not be straightforward.

\begin{lemma}
\label{lem:fiberperm}
\cite[Lemma 18]{MPgc}
Let $(\Delta,\lambda,\iota,\zeta)$ be a cyclic generalised voltage graph, let $n = \lcm\{\iota(\beg x)\lambda(x) \mid x \in \D(\Delta)\}$ and let $\Gamma = \Cyc(\Delta,\lambda,\iota,\zeta)$. For an element $a \in \ZZ_n$ the mapping $f_a:\V(\Gamma) \cup \D(\Gamma) \to \V(\Gamma) \cup \D(\Gamma)$ given by $f_a(x_i)=x_{i+a}$ is an automorphism of $\Gamma$ whose orbits (on vertices and darts) are precisely the fibres of $\Gamma$. Moreover, the group homomorphism $\varphi: \ZZ_n \to \Aut(\Gamma)$ mapping every $a \in \ZZ_n$ to $f_a$ is an embedding.
\end{lemma}

\begin{lemma}
\label{lem:action}
\cite[Lemma 21]{MPgc} Let $(\Delta,\lambda,\iota,\zeta)$ be a cyclic generalised voltage graph and let $n = \lcm\{\iota(\beg x)\lambda(x) \mid x \in \D(\Delta)\}$. For an integer $a$ coprime to $n$ let $\zeta_a(x): \D(\Delta) \to \ZZ_n$ be given by the rule $\zeta_a(x) = \zeta(x)a$. Then $\Cyc(\Delta,\lambda,\iota,\zeta) \cong \Cyc(\Delta,\lambda,\iota,\zeta_a)$.
\end{lemma}

\begin{lemma}
\cite[Lemma 22]{MPgc}
Let $(\Delta,\lambda,\iota,\zeta)$ be a cyclic generalised voltage graph,
let $\Gamma=\Cyc(\Delta,\lambda,\iota,\zeta)$,
let $\varphi \in \Aut(\Delta)$ and let $a$ be an integer coprime to $n$.
If 
$\iota(v^\varphi) = \iota(v)$ for every $v\in \V(\Delta)$ and $\lambda(x^\varphi)=\lambda(x)$, $\zeta(x^\varphi) = a\zeta(x)$ for every $x\in \D(\Delta)$. 
Then the permutation mapping a dart
$ x_i $ to the dart $(x^\varphi)_{ai}$  for every $x \in \D(\Delta)$
extends to an automorphism of $\Gamma$.
\label{lem:coprimemulti}
\end{lemma}

\begin{theorem}
\cite[Theorem 26]{MPgc}
\label{theo:tnormalised}
Let $(\Delta,\lambda,\iota,\zeta)$ be a cyclic voltage graph and let $T$ be a spanning tree of $\Delta$. Then there exists a voltage assignment $\zeta'$ such that $\zeta'(x)=0$ for all $x \in \D(T)$ and $\Cyc(\Delta,\lambda,\iota,\zeta) \cong \Cyc(\Delta,\lambda,\iota,\zeta')$.
\end{theorem}

\begin{definition}
A voltage assignment in which $\zeta(x) = 0$  for all darts $x$ belonging to a prescribed spanning tree $T$ (as $\zeta'$ is in the lemma above) is said to be $T$-normalised.
\end{definition}

\begin{theorem}
\cite[Theorem 32]{MPgc}
\label{theo:concyc}
Let $(\Delta,\lambda,\iota,\zeta)$ be a cyclic generalised voltage graph where $\zeta$ is $T$-normalised for some spanning tree $T$ of $\Delta$. Let $A=\{\zeta(x) | x \in \D(\Delta)\}$ and $B=\{\iota(v) | v \in \V(\Delta)\}$. Then, $\Cyc(\Delta,\lambda,\iota,\zeta)$ is connected if and only if $\gcd(A \cup B)=1$.
\end{theorem}

\begin{proof}
Let $(\Delta,\ZZ_n,\omega,\zeta) = \Psi(\Delta,\lambda,\iota,\zeta)$ and let $\bar{\Gamma}=\GC(\Delta,\ZZ_n,\omega,\zeta)$. Then $\Gamma$ is connected if and only if $\bar{\Gamma}$ is connected. By Theorem 32 of \cite{MPgc}, $\bar{\Gamma}$ is connected if and only if $\ZZ_n = \langle A, \bar{B}\rangle$ where $\bar{B} = \{\omega(v) | v \in \V(\Delta)\}$ and  $n = \lcm\{\iota(\beg x)\lambda(x) \mid x \in \D(\Delta)\}$. However, $\omega(v)$ is precisely the subgroup of $\ZZ_n$ generated by $\iota(v)$. Hence $\bar{\Gamma}$ is connected if and only if $\ZZ_n = \langle A, B \rangle$. This equality, in turn, holds if and only if $\gcd(\{n\} \cup A \cup B) = 1$. Since $n$ is a multiple of $\iota(v)$ for all $v \in \V(\Delta)$, we see that $\gcd(\{n\} \cup A \cup B) = \gcd(A \cup B)$. This completes the proof. 
\end{proof}


\begin{theorem}
\cite[Theorem 37]{MPgc}
\label{prop:simple}
Let $(\Delta,\lambda,\iota,\zeta)$ be a cyclic generalised voltage graph and let $\Gamma=\Cyc(\Delta,\lambda,\iota,\zeta)$. Then $\Gamma$ is a simple graph if and only if for all darts $x\in \D(\Delta)$
all of the following conditions hold:
\begin{enumerate}
\item $\gcd(\lambda(x),\lambda(x^{-1})) = 1$ for all $x \in \D(\Delta)$.
\item $\zeta(x) \not \equiv \zeta(y)$ $(\mod \gcd(\iota(\beg x),\iota(\term x)))$ for any two parallel darts $x,y \in \D(\Delta)$.
\item $\zeta(x) \not \equiv 0$ $(\mod \iota(\beg x))$ for all darts $x$ in a semi-edge.
\end{enumerate}
\end{theorem}

\begin{proof}
Let $(\Delta,\ZZ_n,\omega,\zeta) = \Psi(\Delta,\lambda,\iota,\zeta)$ and let $\bar{\Gamma}=\GC(\Delta,\ZZ_n,\omega,\zeta)$. Then $\Gamma$ is simple if and only if $\bar{\Gamma}$ is simple. Let $x \in \D(\Delta)$. Set $a= \iota(\beg x)$, $b = \iota(\term x)$ and $c = a\lambda(x)$. Note that $a$, $b$ and $c$ are the smallest integers such that $\omega(\beg x) = \langle a \rangle$, $\omega(\term x) = \langle b \rangle$ and $\omega(x) = \langle c \rangle$. By Theorem 37 of \cite{MPgc} (and by the fact that $\ZZ_n$ is abelian), $\bar{\Gamma}$ is simple if and only if the three following conditions hold:
\begin{itemize}
\item[(1')] $h \notin \langle b \rangle$ for all $h \in \langle a \rangle \setminus \langle c \rangle$;         
\item[(2')] $h + \zeta(y) - \zeta(x) \notin \langle c \rangle$
          for all darts $y$, $y\not = x$, which are parallel to $x$ in $\Delta$
          and for all $h\in \langle a \rangle$;
\item[(3')] $\zeta(x)\not \in \langle a \rangle$ if $x=x^{-1}$.
\end{itemize} 
We will show that items (1), (2) and (3) of Theorem \ref{prop:simple} are respectively equivalent to items (1'), (2') and (3') above.

Suppose (1') holds. Then, since $\ZZ_n$ is abelian, for all $h \in \langle a \rangle - \langle c \rangle$ we have $h \notin \langle b \rangle$. From (\ref{eq:volt1}) (and again from the fact that $\ZZ_n$ is abelian), $\langle c \rangle \leq \langle a \rangle \cap \langle b \rangle$ and so $\langle c \rangle = \langle a \rangle \cap \langle b \rangle = \langle \lcm(a,b) \rangle$. This is $c = \lcm(a,b)$.  Now, $c = a\cdot\lambda(x)$ and $b = a\cdot\lambda(x)/\lambda(x^{-1})$ by equation (\ref{eq:ratio}). Hence $c = \lcm(a,b)$ if and only if $ a\cdot\lambda(x) = \lcm(a,a\cdot\lambda(x)/\lambda(x^{-1}))$. The latter equality holds if and only if $\gcd(\lambda(x),\lambda(x^{-1}))=1$. Therefore (1) holds if and only if (1'). 

Now, let $y \in \D(\Delta)$ be parallel to $x$. Suppose (2') holds. Then, for all $h \in \langle a \rangle$, $h + \zeta(y) - \zeta(x) \notin \langle b \rangle$. In particular, $ab \in \langle a \rangle$ and so $b$ cannot divide $\zeta(y) - \zeta(x)$. It follows that $\zeta(x) \not \equiv \zeta(y)$ modulo $\gcd(a,b)$. Finally, recall that $\gcd(a,b) = \gcd(\iota(\beg x),\iota(\term x))$. Thus (2) holds. Conversely, suppose $\zeta(x) \not \equiv \zeta(y)$ $\mod (\gcd(a,b))$. If for some for some $h \in \langle a \rangle$ we have $h + \zeta(y) - \zeta(x) \in \langle b \rangle$, then, since $\gcd(a,b)$ divides both $h$ and $b$, $\gcd(a,b)$ must also divide $\zeta(y) - \zeta(x)$, a contradiction. Thus (2') holds. We conclude that (2) is equivalent to (2').

Finally, it is a plain observation that (3) is equivalent to (3'). 
\end{proof}

\section{ccv-graphs and the proof of Theorem~\ref{the:main}}
\label{sec:ccv}

\subsection{ccv-graphs}

Recall that a ccv-graph is a cyclic generalised voltage graph $(\Delta,\lambda,\iota,\zeta)$ such that its covering graph  is simple, connected and cubic; see Definition~\ref{def:ccv}. In view of Theorem~\ref{theo:tnormalised}, we will assume that $\zeta$ is $T$-normalised for some spanning tree $T$.
Furthermore, in light of the results in the preceding section, a cyclic generalised voltage graph $(\Delta,\lambda,\iota,\zeta)$  is a $\ccv$-graph if and only if $\zeta$ and $\iota$ agree with Theorems \ref{theo:concyc} and \ref{prop:simple}, and if for all $v \in \V(\Delta)$ we have 
\begin{align}
\label{eq:cubic}
\sum\limits_{x \in \Delta(v)} \lambda(x) =3;
\end{align} 

see formula \ref{eq:degree}. Let $(\Delta,\lambda)$ be a dart-labelled graph and recall that $(\Delta,\lambda)$ can be extended to a cyclic generalised voltage graph if there are functions $\iota$ and $\zeta$ satisfying (\ref{eq:ratio}) and (\ref{eq:invvolt}), respectively. It is a plain observation that the cyclic generalised voltage graph $(\Delta,\lambda, \iota, \zeta)$ thus obtained might not be a ccv-graph (this is, the covering graph $\Cyc(\Delta,\lambda,\iota,\zeta)$ need not be simple, connected or cubic). Then, in order to prove Theorem \ref{the:main}, we must determine sufficient and necessary conditions on the dart-labelling $\lambda$ for $(\Delta,\lambda)$ to be extendable to a ccv-graph. 

\begin{proposition}
\label{prop:consistentccv}
Let $(\Delta,\lambda)$ be a connected dart-labelled graph and suppose it is extendable. Then $(\Delta,\lambda)$ can be extended to a ccv-graph $(\Delta,\lambda,\iota,\zeta)$ if and only if the following holds:
\begin{enumerate} 

\item $\sum\limits_{x \in \Delta(v)} \lambda(x) = 3$ for all vertices $v \in \V(\Delta)$.
\item $\lambda(x)=\lambda(x^{-1})$ implies $\lambda(x)=1$.
\item $\lambda(x) = \lambda(y) = 1$ for any two parallel darts $x$ and $y$.
\item $\lambda(x) = 1$ for every dart $x$ underlying a semi-edge.

\end{enumerate}
\end{proposition}

\begin{proof}
Suppose there are functions $\iota$ and $\zeta$ such that $(\Delta,\lambda,\iota,\zeta)$ is a ccv-graph. Observe that (1) holds by formula (\ref{eq:cubic}). Now suppose $\lambda(x)=\lambda(x^{-1})$ for some $x \in \D(\Delta)$. By item (1) of Theorem \ref{prop:simple}, we have $1 = \gcd(\lambda(x),\lambda(x^{-1})) = \lambda(x)$. Hence, (2) holds. Now suppose $x,y \in \D(\Delta)$ are two parallel darts. By (1), $\lambda(x) + \lambda(y) \leq 3$, so either $\lambda(x) = \lambda(y) = 1$ or $\lambda(x) = 1$ and $\lambda(y) = 2$. In the latter case, from (\ref{eq:ratio}) we have $\lambda(y^{-1}) = 2\lambda(x^{-1})$. Then $\lambda(y^{-1})$ is even and $\lambda(y^{-1}) \leq 3$. This is $\lambda(y^{-1})=2=\lambda(y)$, contradicting item (2). Then (3) holds. Finally, item (4) follows at once from item (1) of Proposition \ref{prop:simple}.

For the converse suppose $\lambda$ satisfies conditions (1)--(4). Suppose $\Delta$ is a dipole with vertices $u$ and $v$. Let $\iota(v)=\iota(u)=3$, and let $\{x_0,x_1,x_2\}$ be the set of darts incident to $u$. By letting $\zeta(x_i) = i$ and $\zeta(x_i^{-1}) = -\zeta(x_i)$ for all $x_i \in \{x_0,x_1,x_2\}$, we obtain a ccv-graph $(\Delta,\lambda,\iota,\zeta)$. 

Now, suppose $\Delta$ is not a dipole. Let $\iota$ be the index function given by (\ref{eq:ext1}) and (\ref{eq:ext2}) in the proof of Lemma \ref{lem:consistent}. Let $c$ be the smallest positive integer such that $c\cdot\iota(v)$ is even if $v$ is incident to a semi-edge or to a pair of parallel links, and $c\cdot\iota(v)\geq 3$ if $v$ is incident to a loop. Define $\overline{\iota}: \V(\Delta) \to \ZZ$ by $\lin{\iota}(v) = c\cdot\iota(v)$ for all $v \in \V(\Delta)$. 

To define an appropriate voltage assignment, let $n=\lcm\{\lin{\iota}(\beg)\lambda(x) \mid x\in\D(\Delta)\}$ and let $\Delta'$ be a maximal simple subgraph of $\Delta$ (note that $\Delta'$ always exists and can be found by removing all darts, all loops and one edge from each pair of parallel edges in $\Delta$). Let $D^*$ be a subset of $\D(\Delta)$ containing exactly one dart for each edge of $\Delta$. Define $\zeta^*: D^* \to \ZZ_n$ as follows: 

$$
\zeta^*(x) = \left\{
        \begin{array}{ll}        	
            0 & \quad \hbox{ if $x \in \D(\Delta')$};\\
            \iota(\beg x)/2 & \quad \hbox{ if $x$ underlies a semi-edge};\\
            1 & \quad \hbox { otherwise}.
        \end{array}
    \right.
$$

We can now extend $\zeta^*$ to a voltage assignment $\zeta$ by letting $\zeta(x) = \zeta^*(x)$ if $x \in D^*$ and $\zeta(x) = \zeta^*(x^{-1})^{-1}$ if $x \notin D^*$. It is straightforward to see that $\zeta$ satisfies (\ref{eq:invvolt}) and that it is $T$-normalised. Therefore, $(\Delta,\lambda,\lin{\iota},\zeta)$ is a cyclic generalised voltage graph. We leave to the reader to verify that $(\Delta,\lambda,\lin{\iota},\zeta)$ is indeed a ccv-graph.
\end{proof}

Let $(\Delta,\lambda,\iota,\zeta)$ be a $\ccv$-graph and let $\{x,x^{-1}\}$ be an edge of $\Delta$. Then we say $xx^{-1}$ is an edge of type $[\lambda(x),\lambda(x^{-1})]$, or simply a $[\lambda(x),\lambda(x^{-1})]$-edge (note that the order in which we write $xx^{-1}$ is important). If $\beg x = u$ and $\beg x^{-1}=v$ then we will also say, when there is no possibility of ambiguity, that $uv$ is a $[\lambda(x),\lambda(x^{-1})]$-edge. The following is a direct consequence of Proposition~\ref{prop:consistentccv}, 

\begin{corollary}
A $\ccv$-graph only has edges of type $[1,1]$, $[1,2]$, $[1,3]$ and $[2,3]$. In particular, parallel edges in a ccv-graph are necessarily $[1,1]$-edges. 
\end{corollary}

Recall that we will always assume a voltage assignment to be trivial on a prescribed spanning tree of $\Delta$; see Theorem \ref{theo:tnormalised}. In the particular case of ccv-graphs 
we can chose this tree so that it contains all $[i,j]$-edges with $i \neq j$. Even further assumptions, stated in Lemma \ref{lem:simplified} at the end of this subsection, can be made about $\zeta$, so that $\zeta$ is "as nice as possible".

We say a dart-labelled graph $(\Delta',\lambda')$ is a subgraph of $(\Delta,\lambda)$ if $\Delta' < \Delta$ and $\lambda'$ is the restriction of $\lambda$ to $\D(\Delta')$. 

\begin{proposition}
\label{prop:tree}
Let $(\Delta, \lambda, \iota, \zeta)$ be a $\ccv$-graph. Then $(\Delta, \lambda, \iota, \zeta)$ has a spanning tree $\mathcal{T}$ that contains all the $[i,j]$-edges with $i \neq j$.
\end{proposition}

\begin{proof}
Let $(\Delta',\lambda')$ be the subgraph of $(\Delta,\lambda)$ induced by all $[i,j]$-edges of $(\Delta, \lambda, \iota, \zeta)$ with $i \neq j$. We will show that $\Delta'$ is acyclic. Suppose to the contrary that $\Delta'$ contains a cycle $C$. Let $u$ be the the vertex of $C$ with the least index, that is $\iota(u) \leq \iota(v)$ for all $v \in \V(C)$. Let $v$ and $w$ be the two neighbours of $u$ in $C$. Now, $uv$ cannot be a $[1,1]$-edge as $\Delta'$ contains no such edges. If $uv$ is a $[1,j]$-edge with $j \in \{2,3\}$, then $\iota(v)=\frac{1}{j}\cdot\iota(u)<\iota(u)$, which contradicts $u$ having minimal index. Hence, $vu$ must be a $[1,j]$-edge with $j \in \{2,3\}$. Similarly, $wu$ is a $[1,j]$-edge with $j \in \{2,3\}$. This implies that the sum of the labels of darts in $\Delta(u)$ is at least $4$, contradicting formula (\ref{eq:cubic}). Hence, $\Delta'$ is acyclic and thus can always be completed to a spanning tree $\mathcal{T}$ of $\Delta$.
\end{proof}

\begin{lemma}
\label{lem:simplified}
Let $(\Delta,\lambda,\iota,\zeta)$ be a ccv-graph. Then there exists a voltage assignment $\zeta'$ such that $\Cyc(\Delta, \lambda,\iota,\zeta) \cong \Cyc(\Delta, \lambda,\iota,\zeta')$ and for all $x \in \D(\Delta)$, the following holds:
\begin{enumerate}
\item $\zeta'(x) < \gcd(\iota(\beg x),\iota(\term x))$ for all $x \in \D(\Delta)$;
\item $\zeta'$ is $T$-normalised for a spanning tree of $\Delta$ containing all $[i,j]$-edges with $i \neq j$;
\item $\zeta'(x) = \iota(\beg x)/2$ whenever $x$ underlies a semi-edge;
\item $0 < \zeta'(x) < \iota(\beg x)$ and $\zeta'(x) \neq \iota(\beg x)/2$ whenever $x$ underlies a loop.
\end{enumerate}  
\end{lemma}

\begin{proof}
By Proposition \ref{prop:tree}, $(\Delta,\lambda,\iota,\zeta)$ admits a spanning tree $T$ that contains all $[i,j]$-edges with $i\neq j$. By Theorem \ref{theo:tnormalised} we can assume without loss of generality that $\zeta$ is $T$-normalised. Let $\zeta'$ be the voltage obtained by reducing $\zeta(x)$ modulo $\gcd(\iota(\beg x),\iota(\term x))$ for all $x \in \D(\Delta)$. By Remark~\ref{rem:reduced}, $\Cyc(\Delta, \lambda,\iota,\zeta) = \Cyc(\Delta, \lambda,\iota,\zeta')$. Items (1) and (2) follow at once from our choice of $\zeta'$. 

Suppose a dart $x \in \D(\Delta)$ underlies a semi-edge. Then $\zeta'(x) \neq 0$ by Theorem~\ref{prop:simple}. Moreover $\zeta'(x)=\zeta'(x^{-1})$ since $x = x^{-1}$. Then by equality (\ref{eq:invvolt}) we have $\zeta'(x) \equiv -\zeta'(x)$ ($\mod \iota(\beg x)$), or equivalently, $2\cdot\zeta'(x) \equiv 0$ ($\mod \iota(\beg x)$). Since $0 < \zeta'(x) < \iota(\beg x)$, we see that $\iota(\beg x)$ is even and $\zeta'(x) = \iota(\beg x)/2$. Hence (3) holds.

Finally, if $x$ underlies a loop, by Theorem~\ref{prop:simple}, $\zeta'(x) \not \equiv \zeta'(x^{-1})$ ($\mod \iota(\beg x)$) and therefore $0 \neq \zeta'(x) \neq \iota(\beg x)/2$. Hence (4) holds. This completes the proof.
\end{proof}

We call a voltage assignment like in Lemma \ref{lem:simplified} a {\em simplified voltage assignment}. Since every simple, connected, cubic graph is the cyclic generalised cover of a ccv-graph with a simplified voltage assignment, we will henceforth always assume that the voltage of a ccv-graph is simplified.

\subsection{Proof of Theorem \ref{the:main}}
\label{ssec:proof}
We would like to begin by making a few observations about the graphs $\Delta_i$ appearing in Figure \ref{fig:quotients}. These $25$ graphs comprise the complete list of graphs on at most $3$ vertices with a dart-labelling that agrees with Proposition \ref{prop:consistentccv}. Although it is a simple enough computation to be done by hand, in order to avoid human error, we used a computer programme written in SAGE \cite{sage} to construct, by brute force, all such graphs (up to label-preserving isomorphism). Then, any ccv-graph on at most $3$ vertices can be obtained by extending some $\Delta_i$ to a cyclic generalised voltage graph. 
In light of Lemma \ref{lem:simplified}, when extending a labelled graph $\Delta_i$ to a ccv-graph, we will only consider those voltage assignments that agree with the corresponding voltage shown in Figure \ref{fig:quotients} (that is, a voltage assignment that is trivial on every edge lacking an arrowhead, and that assigns voltage $\iota(\beg x)/2$ to every semi-edge $x$). The conditions under each $\Delta_i$ are derived from equalities (\ref{eq:ratio}) and (\ref{eq:invvolt}), and from Theorems \ref{theo:concyc} and \ref{prop:simple}. Hence, an extension $(\Delta,\lambda,\iota,\zeta)$ of $\Delta_i$, for some $i \in \{1,\ldots,25\}$, is a ccv-graph if and only if $\iota$ and $\zeta$ satisfy the corresponding conditions listed in Theorem \ref{the:main} (and agree with the corresponding voltage assignment shown in Figure \ref{fig:quotients}). 

Let $\Gamma$ be a connected, cubic, simple graph and suppose it admits a cyclic group of automorphism $G$ having at most $3$ vertex-orbits. By Theorem \ref{the:lift}, $\Gamma$ is the generalised cover of a generalised voltage graph $(\Gamma/G,G,\omega,\zeta)$, where the quotient $\Gamma/G$ has at most $3$ vertices. Since $G$ is a cyclic group, by Theorem \ref{theo:corresp}, $\Gamma \cong \Cyc(\Gamma/G,\lambda,\iota,\zeta)$ where $(\Gamma/G,\lambda,\iota,\zeta) = \Phi(\Gamma/G,G,\omega,\zeta)$. Moreover, since $G$ is connected, cubic and simple $(\Gamma/G,\lambda,\iota,\zeta)$ is a ccv-graph. Then $(\Gamma/G,\lambda,\iota,\zeta)$ is an extension of $\Delta_i$, for some $i \in \{1,\dots,25\}$. That is $\Gamma$ is isomorphic to $\Gamma_i(m;r,s)$ (or $\Gamma(m;r)$ or $\Gamma(m)$) for some $m$, $r$ and $s$ satisfying the corresponding conditions listed in Theorem \ref{the:main}.

For the converse, let $(\Delta,\lambda,\iota,\zeta)$ be a ccv-graph obtained by extending one dart-labelled graph $\Delta_i$ from Figure \ref{fig:quotients}. Recall that the generalised cover of a ccv-graph is a simple, cubic, connected graph. Moreover, the group $\ZZ_n$, where $n = \lcm\{\lambda(x)\iota(\beg x) \mid x \in \D(\Delta_i)\}$, acts as a group of automorphism of $\Cyc(\Delta,\lambda,\iota,\zeta)$, and the orbits on vertices and darts under this action are precisely the fibres of $\Cyc(\Delta,\lambda,\iota,\zeta)$ (see Lemma \ref{lem:fiberperm}). This completes the proof of Theorem~\ref{the:main}.

\section{Cubic vertex-transitive graphs with a cyclic group having at most $3$ orbits}
\label{sec:vt}

We devote this section to the proof of Theorem \ref{the:vt}
stated in Section~\ref{sec:intro}.
That is, 
 we will show that a connected cubic vertex-transitive graph admitting a cyclic group of automorphisms  having at most $3$ vertex-orbits 
is either isomorphic to the Tutte-Coxeter graph $\Gamma_{25}(10,1,3)$ 
 or
it belongs to one the following $5$ infinite families:
 \begin{enumerate}
\item $\Gamma_1(m;r)$ with $m \equiv 0$ $(\mod 2)$, $m \geq 4$, $\gcd(\frac{m}{2},r)=1$, $r \in \{1,2\}$;
\item $\Gamma_2(m;r,1)$ with $m \geq 3$, $r^2 \equiv \pm 1$ $(\mod m)$, or $m = 10$ and $r=2$;
\item $\Gamma_4(m;r,s)$ with $m \geq 3$, $r \neq s$, $\gcd(m,r,s)=1$;
\item $\Gamma_{22}(m;2,1)$ with $m \geq 4$ and $\frac{m}{2}$ odd;
\item $\Gamma_{23}(m;r,1)$ with $\frac{m}{2} \equiv 1$ $(\mod 4)$ and $r = (\frac{m}{2}+3)/2$,  or $\frac{m}{2} \equiv 3$ $(\mod 4)$ and $r = (\frac{3m}{2}+3)/2$, or $m=4$ and $r=0$. 
\end{enumerate} 

Let $E$ be the set containing the following $9$ exceptional graphs: $K_4$, $K_{3,3}$, the cube $Q_3$, the Petersen graph, the Heawood graph, the generalised Petersen graph $\textrm{GP}(8,3)$, the Pappus graph, the dodecahedron $\textrm{GP}(10,2)$ and the generalised Petersen graph $\textrm{GP}(10,3)$ (see \cite{girth7} for definitions and properties of the Pappus graph, the Heawood graph and the generalised Petersen graphs).

We will partition the set $\{1,\ldots,25\}$ into three sets: $I_M = \{1,2,3,4,22,23,24,25\}$, $I_E = \{5,6,8,16,18,19,21\}$ and $I_C = \{7,9,10,11,12,13,14,15,17,20\}$. The reason for this is that a ccv-cover of $\Delta_i$ with $i \in I_M$ is necessarily a $k$-multicirculant for some $k \in \{1,2,3\}$. We can then resort to the classification of cubic vertex-transitive bicirculants ($2$-multicirculants) \cite{bic} and tricirculants ($3$-multicirculant) \cite{MPtricirc} to deal with ccv-covers of $\Delta_i$ when $i \in I_M$. For the remaining indices in $\{1,\ldots,25\} \setminus I_M$, we distinguish those for which $\Delta_i$ admits no vertex-transitive ccv-cover (these conform the set $I_C$) and those who admit at least one vertex-transitive ccv-cover (these conform the set $I_E$). As it transpires, all vertex-transitive ccv-covers of $\Delta_i$ with $i \in I_E$ belong to the set $E$ of exceptional graphs. Theorem~\ref{the:vt} will then follow from Claims \ref{cl:ex} and \ref{cl:noVT} (which will be proved at the end of this section), as well as the classification of cubic vertex-transitive bicirculants and tricirculants.

\begin{claim}
\label{cl:ex}
If $i \in I_E$, then a vertex-transitive ccv-cover of $\Delta_i$ must be one of the $9$ exceptional graphs in $E$. 
\end{claim}

\begin{claim}
\label{cl:noVT}
If $i \in I_C$ then $\Delta_i$ admits no vertex-transitive ccv-cover.
\end{claim}

We will now introduce the concept of the {\em signature} of a cubic graph, which we will need to prove the claims above. Let $\Gamma$ be a simple cubic graph. For an edge $e$ of $\Gamma$ and a positive integer $c$ denote by $\epsilon_c(e)$ the number of $c$-cycles (cycles of length $c$) that pass through $e$. For a vertex $v$ of $\Gamma$, let $\{e_1,e_2,e_3\}$ be the set of edges incident to $v$ ordered in such a way that $\epsilon_c(e_1) \leq \epsilon_c(e_2) \leq \epsilon_c(e_3)$. The triplet $(\epsilon_c(e_1),\epsilon_c(e_2),\epsilon_c(e_3))$ is then called the {\em $c$-signature} of $v$. If for a positive integer $c$ all vertices of $\Gamma$ have the same $c$-signature, we say that $\Gamma$ is {\em $c$-cycle-regular}. In particular, if $\Gamma$ is $g$-cycle-regular, where $g$ is the girth of $\Gamma$, we will say that $\Gamma$ is {\em girth-regular}, following the nomenclature of \cite{signature}. Observe that if $\Gamma$ is vertex-transitive, then $\Gamma$ is $c$-cycle-regular for all $c \in \NN$. 

We will need the five following lemmas. Lemma \ref{lem:girthregular} is proved in \cite{signature}. Lemma \ref{lem:smallgirth1} is almost folklore, it is mentioned in \cite{smallgirth} and \cite{girth7} but a direct proof is not provided. It follows as a corollary of Theorem 5 of \cite{signature}. Meanwhile, Lemma \ref{lem:smallgirth2} is proved in \cite{girth7}. We provide full proof of Lemmas \ref{lem:loop} and \ref{lem:cycles}.

\begin{lemma}
\cite[Theorem 5]{signature}
\label{lem:girthregular}
Let $\Gamma$ be a cubic girth-regular graph of girth $g \leq 5$. Then either the $g$-signature of $\Gamma$ is $(0,1,1)$ or one of the following occurs:
\begin{enumerate}
\item $g=3$ and $\Gamma \cong K_4$;
\item $g=4$ and $\Gamma$ is isomorphic to a prism or a M\"{o}bius ladder; 
\item $g=5$ and $\Gamma$ is isomorphic to the Petersen graph or the Dodecahedron.
\end{enumerate}
\end{lemma}

\begin{lemma}
\label{lem:smallgirth1}
If $\Gamma$ is a cubic arc-transitive graph of girth smaller than $6$, then $\Gamma$ is isomorphic to one of the following: $K_4$, $K_{3,3}$, the three-dimensional cube $Q_3$, the Petersen Graph or the Dodecahedron.
\end{lemma}

\begin{lemma}
\label{lem:smallgirth2}
\cite[Proposition 4.2]{girth7}
If $\Gamma$ is a cubic arc-transitive graph of girth $6$ then every edge of $\Gamma$ lies in exactly $2$ cycles of length $6$, or $\Gamma$ is isomorphic to one of the following graphs: the Heawood graph, the Pappus graph or the generalised Petersen graph $\rm{GP}(i,3)$ with $i = 8,10$.
\end{lemma}

Before we state and prove Lemma \ref{lem:loop}, we need to look briefly into the set of exceptional graphs $E$ and the possible dart-labelled graphs from which they arise. This will also be used in the proofs of Claims \ref{cl:ex} and \ref{cl:noVT}. For each graph $\Gamma$ in $E$, we have determined, with the aid of a computer programme written in SAGE \cite{sage}, the values of $i\in\{1,\ldots,25\}$ for which $\Gamma$ is a cyclic generalised cover of $\Delta_i$. For instance, $K_4$ is a cyclic generalised cover of $\Delta_1$ and $\Delta_3$ (and of no other $\Delta_i$). Indeed, $K_4$ is isomorphic to $\Gamma_1(4;1)$ and to $\Gamma_3(2;1)$. The results are displayed in Table \ref{tab:quo} below.

\begin{table}[H]
\begin{center}
\begin{tabular}{l l l l}
\hline
Graph & Cover of $\Delta_i$ & Graph & Cover of $\Delta_i$\\
\hline
$K_4$ & $i = 1,3$ &  $\textrm{GP}(8,3)$ & $i = 2,6$\\
$K_{3,3}$ & $i= 1,4,5,16$ \qquad \qquad \qquad & Pappus & $i = 5,22$\\
$Q_3$ & $i = 2,4,6$ & $\textrm{GP}(10,2)$ & $i =2$\\
Petersen & $i = 2,19,21$ & $\textrm{GP}(10,3)$ & $i=2$\\
Heawood &  $i = 4,13,18$ & & \\
\hline
\end{tabular}
\caption{The exceptional graphs in $E$ and the values of $i$ for which they are a ccv-cover of $\Delta_i$}
\label{tab:quo}
\end{center} 
\end{table}

\begin{remark}
\label{rem:E}
Each graph in $E$ is a ccv-cover of $\Delta_i$ for some $i \in I_M$, and thus is a $k$-multicirculant for some $i \in \{1,2,3\}$.
\end{remark}

\begin{remark}
\label{rem:label3}
Let $(\Delta, \lambda, \iota, \zeta)$ be a ccv-graph with an edge of type $[1,3]$ or $[2,3]$. If $\Gamma:=\Cyc(\Delta, \lambda, \iota, \zeta)$ is vertex-transitive, then it is arc-transitive. Indeed, suppose $\lambda(x)=3$ for some $x \in \D(\Delta)$, and let $u = \beg x$. Then all three darts in $\fib(x)$ beginning at a fixed vertex $u_i \in \fib(u)$ belong to the same orbit under the action of $\Aut(\Gamma)$ (recall that $\Aut(\Gamma)$ contains a copy of $\ZZ_n$, with $n = \lcm\{\iota(\beg x)\lambda(x) \mid x \in \D(\Delta)\}$, that acts transitively on the fibres of $\Gamma$). Since $\Gamma$ is vertex-transitive, this must be true for all vertices. This is, all the three darts beginning at any vertex belong to the same $\Aut(\Gamma)$-orbit. Thus $\Gamma$ is edge-transitive. Finally, it is well known that graph with odd valency that is both vertex- and edge-transitive must also be arc-transitive.  
\end{remark}

\begin{lemma}
\label{lem:loop}
Let $(\Delta, \lambda, \iota, \zeta)$ be a $\ccv$-graph. Let $u,v \in \V(\Delta)$ be adjacent through a $[1,i]$-edge, $i \in \{2,3\}$ and suppose $u$ is incident to a loop. If $\Gamma:=\Cyc(\Delta, \lambda, \iota, \zeta)$ is vertex-transitive then it isomorphic to the Heawood graph or the Pappus graph.
\end{lemma}

\begin{proof}
First, suppose $uv$ is a $[1,2]$-edge consisting of two darts: $x$ beginning at $v$, and $x^{-1}$ beginning at $u$. Let $\{y,y^{-1}\}$ be a loop incident at $u$. Let $m$ be the number of $\Aut(\Gamma)$-orbits on darts. Since $\lambda(x) = 2$, then for every vertex $u_i \in \fib(u)$ there are precisely two darts in $\fib(x)$ incident to $u_i$. These two darts must belong to the same $\Aut(\Gamma)$-orbit, since $\Gamma$ has a cyclic subgroup of automorphism acting transitively in each fibre. Since $\Gamma$ is vertex transitive, the orbits of $\Aut(\Gamma)$ on darts are precisely the orbits of $\Aut(\Gamma)_{u_i}$, the stabiliser of $u_i$ in $\Aut(\Gamma)$, on its action on the darts beginning at $u_i$. Then, $\Aut(\Gamma)$ has at most $2$ orbits on darts. Suppose $m=2$. Then $\Aut(\Gamma)$ has two orbits on darts, say $O_1$ and $O_2$, and since $\Gamma$ is vertex-transitive, every vertex of $\Gamma$ is incident to precisely one dart in $O_1$ and two darts in $O_2$. Hence $|O_2| = 2\cdot|O_1|$. Moreover, for every dart $z \in \D(\Gamma)$ both $z$ and its inverse $z^{-1}$ belong to the same $\Aut(\Gamma)$-orbit (for otherwise every edge of $\Gamma$ has one dart in $O_1$ and the other in $O_2$, which implies that $|O_1|=|O_2|$, a contradiction). In particular, this implies that both $\fib(y)$ and $\fib(y^{-1})$ are subsets of $O_2$. On the other hand, $\fib(x) \subseteq O_2$ as $\lambda(x)=2$, and so $\fib(x^{-1}) \subseteq O_2$. Then, all three darts incident to a vertex $\bar{u} \in \fib(u)$ lie in the same $\Aut(\Gamma)$-orbit, contradicting that $m=2$. Therefore $m=1$ and $\Gamma$ is arc-transitive.

Let $k = \iota(v)$ (that is, $k = |\fib(v)|$) and recall that $x$ is the dart beginning at $v$ and ending at $u$. Since $\lambda(x)=2$ there must be another dart, say $z$, incident to $v$. Since $uv$ is a $[1,2]$-edge, we see that $\iota(u) = 2k$. Recall that $\zeta(x) = 0$, as $uv$ is a $[1,2]$-edge (see Lemma~\ref{lem:simplified}), and let $\zeta(y)=r$. Notice that for all $i \in \ZZ_{2k}$ and all $j \in \{-1,1\}$ 
\begin{equation}
\label{eq:cycle}
C_{i,j}:=(v_i,u_i,u_{i+jr},v_{i+jr},u_{i+jr +k},u_{i+k}) 
\end{equation}
is a $6$-cycle in $\Gamma$ (see Figure \ref{fig:loop}, left). In particular $v_0u_0$ lies in $2$ distinct $6$-cycles, namely $C_{0,1}$ and $C_{0,-1}$. Furthermore, $v_0u_k$ also lies on both $C_{0,1}$ and $C_{0,-1}$. Since $\Gamma$ is arc-transitive, $z_0 \in \fib(z)$ must lie in a $6$-cycle $C$. It is plain to see that $C$ must visit either $v_0u_0$ or $v_0u_m$. This is, one of the edges incident to $v_0$ lies in at least $3$ distinct $6$-cycles and by Lemma \ref{lem:smallgirth2}, $\Gamma$ is isomorphic to the Heawood graph, the Pappus graph or the generalised Petersen graph $\rm{GP}(i,3)$ with $i = 8,10$. However the generalised Petersen graphs $\rm{GP}(i,3)$ with $i \in \{8,10\}$ are ccv-covers of $\Delta_i$ only when $i=2$ (see Table \ref{tab:quo}) and $\Delta_2$ has no edges of type $[1,2]$. Therefore $\Gamma$ must be isomorphic to the Pappus graph or the Heawood graph.

Now, suppose $uv$ is a $[1,3]$-edge. Then, $\Gamma$ is arc-transitive. If $\iota(v)$ equals $1$ or $2$, $\Gamma$ is isomorphic to the $K_4$ or the cube graph $Q_3$ respectively (see Figure \ref{fig:examples}). If $\iota(v) \geq 3$, the edge $u_0v_0$ lies in $4$ distinct $6$-cycles (see Figure \ref{fig:loop}, right) and the result follows from Lemma \ref{lem:smallgirth2}.
\end{proof}

\begin{figure}[h!]
\includegraphics[width=0.6\textwidth]{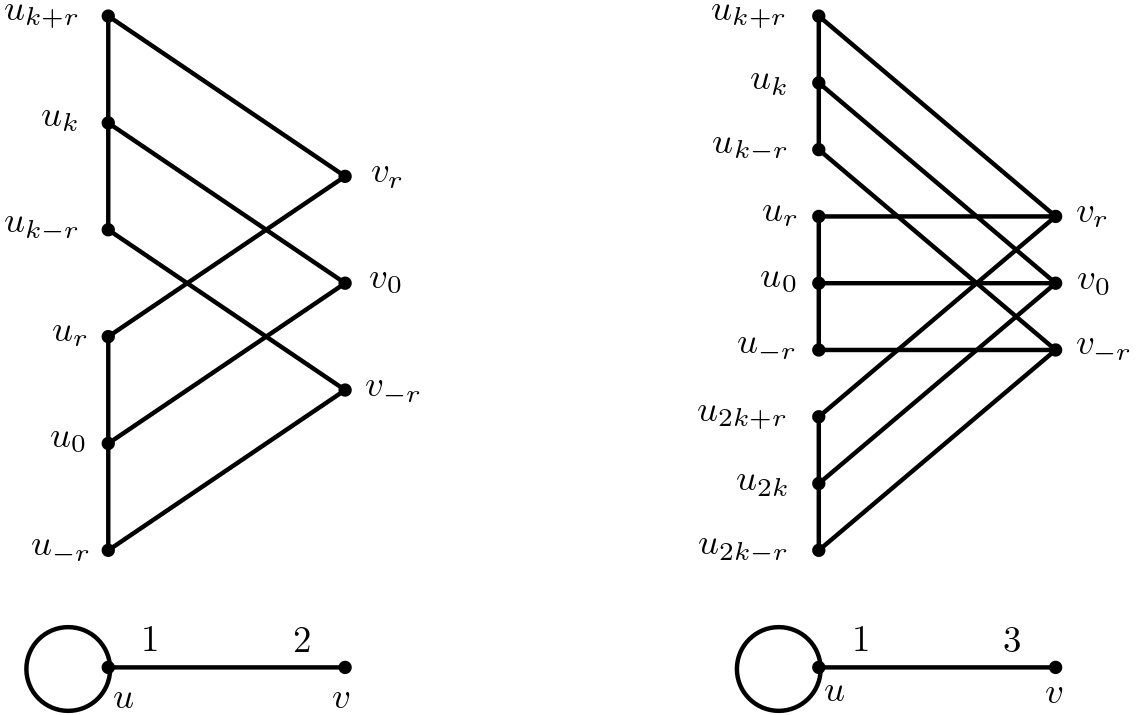}
\caption{A subgraph of $\Gamma$ with vertices in $\fib(u) \cup \fib(v)$ for each of the two cases considered in the proof of Lemma \ref{lem:loop}. Underneath each subgraph, the subgraph of $(\Delta,\lambda)$ to which it projects.}
\label{fig:loop}
\end{figure}

\begin{figure}[h!]
\centering
\includegraphics[width=0.6\textwidth]{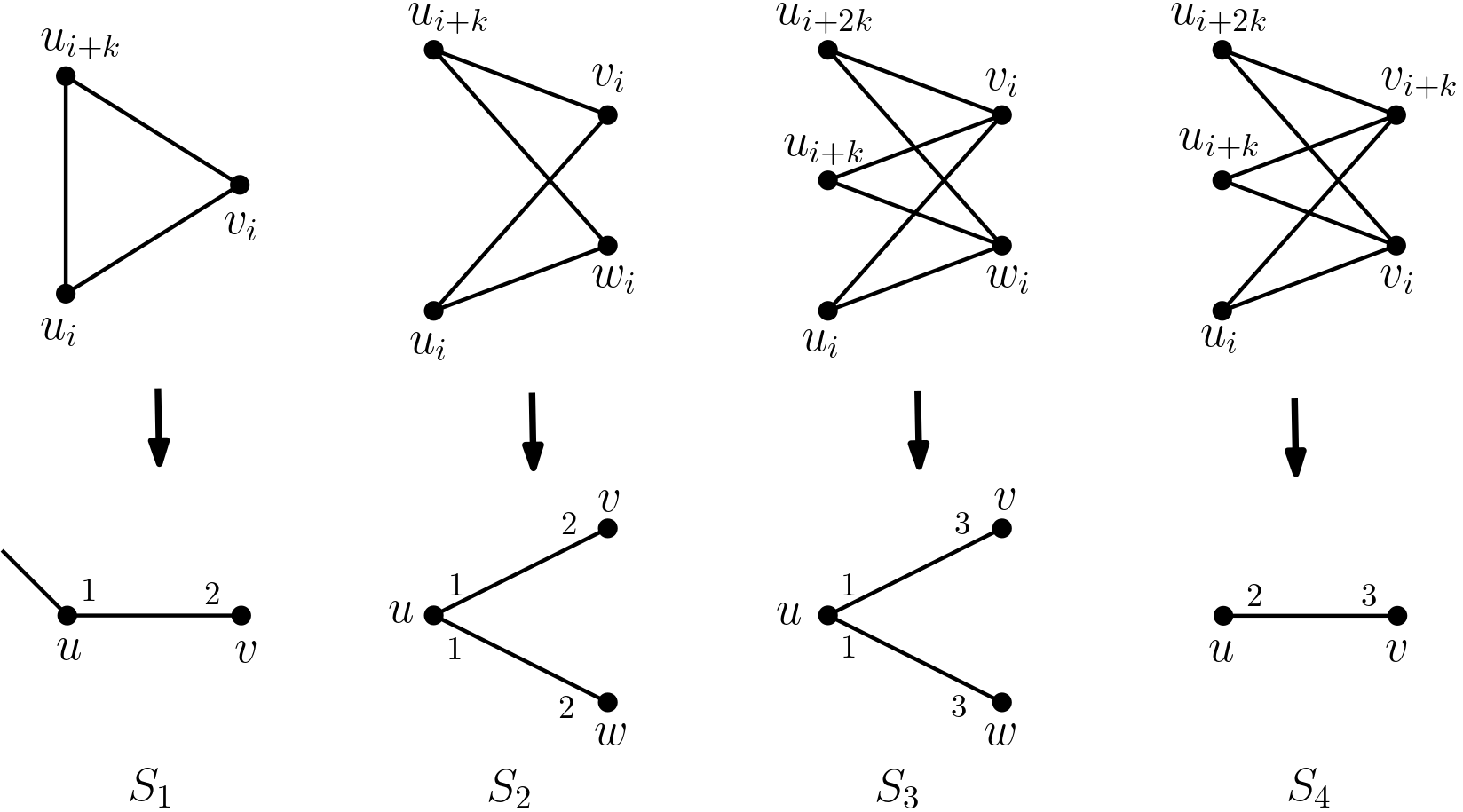}
\caption{In the bottom row, the four subgraphs $S_i$, $i \in \{1,2,3,4\}$. Above each, a connected component of $\pi^{-1}(S_i)$. In the first three cases we assume $\iota(v)=k$ while in the fourth we take $\iota(v)=2k$.}
\label{fig:cycles}
\end{figure}

Let $\Delta$ and $\Delta'$ be two graphs and let $\lambda$ and $\lambda'$ be two dart labellings for $\Delta$ and $\Delta'$ respectively. We say the dart labelled graphs $(\Delta,\lambda)$ and $(\Delta',\lambda')$ are isomorphic if is there is a graph isomorphism $\phi:\Delta \to \Delta'$ such that $\lambda(x)=\lambda'(\phi(x))$. 

For $i \in \{1,2,3,4\}$, let $S_i$ be the corresponding dart-labelled graphs shown in the bottom row of Figure \ref{fig:cycles}, so $S_1$ is a a $[1,2]$ edge with a loop attached to one of its endpoints and so on. Recall that if $(\Delta, \lambda,\iota,\zeta)$ is a ccv-graph and $\Gamma = \Cyc(\Delta, \lambda,\iota,\zeta)$, then the covering projection $\pi$ maps every $x_i \in \fib(x)$ to $x$, for $x \in \D(\Delta) \cup \V(\Delta)$. In addition, recall that if $(\Delta,\lambda)$ and $(\Delta',\lambda')$ are dart-labelled graphs, then say $(\Delta',\lambda')$ is a labelled-subgraph of $(\Delta,\lambda)$  (and we write $(\Delta',\lambda')<(\Delta,\lambda)$) if $\Delta'$ is a subgraph of $\Delta$ and $\lambda'=\lambda\mid_{\D(\Delta')}$.

\begin{lemma}
\label{lem:cycles}
Let $(\Delta, \lambda,\iota,\zeta)$ be a ccv-graph and let $\Gamma = \Cyc(\Delta, \lambda,\iota,\zeta)$. Let $(\Delta',\lambda') \leq (\Delta,\lambda)$ and suppose $(\Delta',\lambda') \cong S_i$ for some $i \in \{1,2,3,4\}$.
\begin{enumerate}
\item If $i =1 $, then $\pi^{-1}(\Delta',\lambda')$ is a union of $3$-cycles.
\item If $i =2$, then $\pi^{-1}(\Delta',\lambda')$ is a union of $4$-cycles.
\item If $i \in \{3,4\}$ then $\pi^{-1}(\Delta',\lambda')$ is a union of copies of $K_{3,2}$.
\end{enumerate}
\end{lemma}

\begin{proof}
Suppose $S_i \cong (\Delta',\lambda')$. If $i=1$ then $(\Delta,\lambda)$ contains a $[1,2]$-edge, say $uv$, where $u$ is incident to a semi-edge. Recall that all $[i,j]$-edges with $i \neq j$ have trivial voltage in a ccv-graph. Then by Lemma~\ref{lem:adjacency}, every $v_i \in \fib(v)$ is adjacent to $u_i$ and $u_{i+k}$ where $k = \iota(v)$. Moreover, every $u_i \in \fib(u)$ is adjacent to $u_{i+k}$. Hence for all $i \in \ZZ_{\iota(v)}$, $(v_i,u_i,u_{i+k})$ is a $3$-cycle (see Figure~\ref{fig:cycles}). It is plain to see that $\pi^{-1}(\Delta',\lambda')$ is the union of all such $3$-cycles. The cases when $i$ equals $2$, $3$ and $4$ follow from an analogous argument.
\end{proof}


{\bf Proof of Claim \ref{cl:ex}.}
Let $ i \in I_E:=\{5,6,8,16,18,19,21\}$ and suppose $\Gamma$ is a vertex-transitive ccv-cover of $\Delta_i$. 
If $i \in \{5,6,18,19\}$ then by Lemma \ref{lem:loop} we have $\Gamma \in E$. If $i = 16$ then  by Lemma \ref{lem:cycles} the girth of $\Gamma$ is at most $4$ and by Remark \ref{rem:label3} $\Gamma$ is arc-transitive. Then, by Lemma \ref{lem:smallgirth1}, $\Gamma \in E$. If $i = 8$, then by Lemma \ref{lem:cycles} $\Gamma$ has $3$-signature $(\epsilon_1,\epsilon_2,\epsilon_3)$, where $\epsilon_1, \epsilon_2, \epsilon_3 > 0$, and by Lemma \ref{lem:girthregular} $\Gamma$ is isomorphic to $K_4$, which is in $E$. It remains to see what happens when $i = 21$.
Suppose $\Gamma$ is a vertex-transitive ccv-cover of $\Delta_{21}$. Then, by Theorem \ref{the:main}, $\Gamma$ is isomorphic to $\Gamma_{21}(m;r)$ for two integer $m,r>0$ such that $m$ is even and $\gcd(\frac{m}{2},r)=1$. With the notation of Figure \ref{fig:cases} (right), let $\iota(w)=m$ and let $g$ be the girth of $\Gamma_{21}(m;r)$. Observe that the vertex $u_0$ lies on two distinct $5$-cycles, namely $(w_0,u_0,v_0,v_m,u_m)$ and $(w_0,u_0,v_r,v_{r+m},u_m)$. This is, every edge incident to $u_0$ lies on a $5$-cycle. If $g=5$ then by Lemma \ref{lem:girthregular}, $\Gamma_{21}(m;r)$ is isomorphic to the Petersen graph or the dodecahedron, and we are done since both graphs are in $E$. Clearly, $g \neq 3$ as no vertex in $\fib(w)$ lies on a $3$-cycle (vertices in $\fib(w)$ are only adjacent to vertices in $\fib(u)$, which is an independent set). If $g=4$ then $w$ must lie on a $4$-cycle. It is not too difficult to see that this can only happen if $v_m = v_r$ and $v_0 = v_{r+m}$. However, $(u_0,v_m,v_0)$ is then a $3$-cycle, contradicting $g \neq 3$. We conclude that if $i = 21$, then $\Gamma \in E$. We have shown that if $i \in I_E$, then a vertex-transitive ccv-cover of $\Delta_i$ must be an exceptional graph belonging to $E$. 

\begin{figure}[h!]
\includegraphics[width=0.7\textwidth]{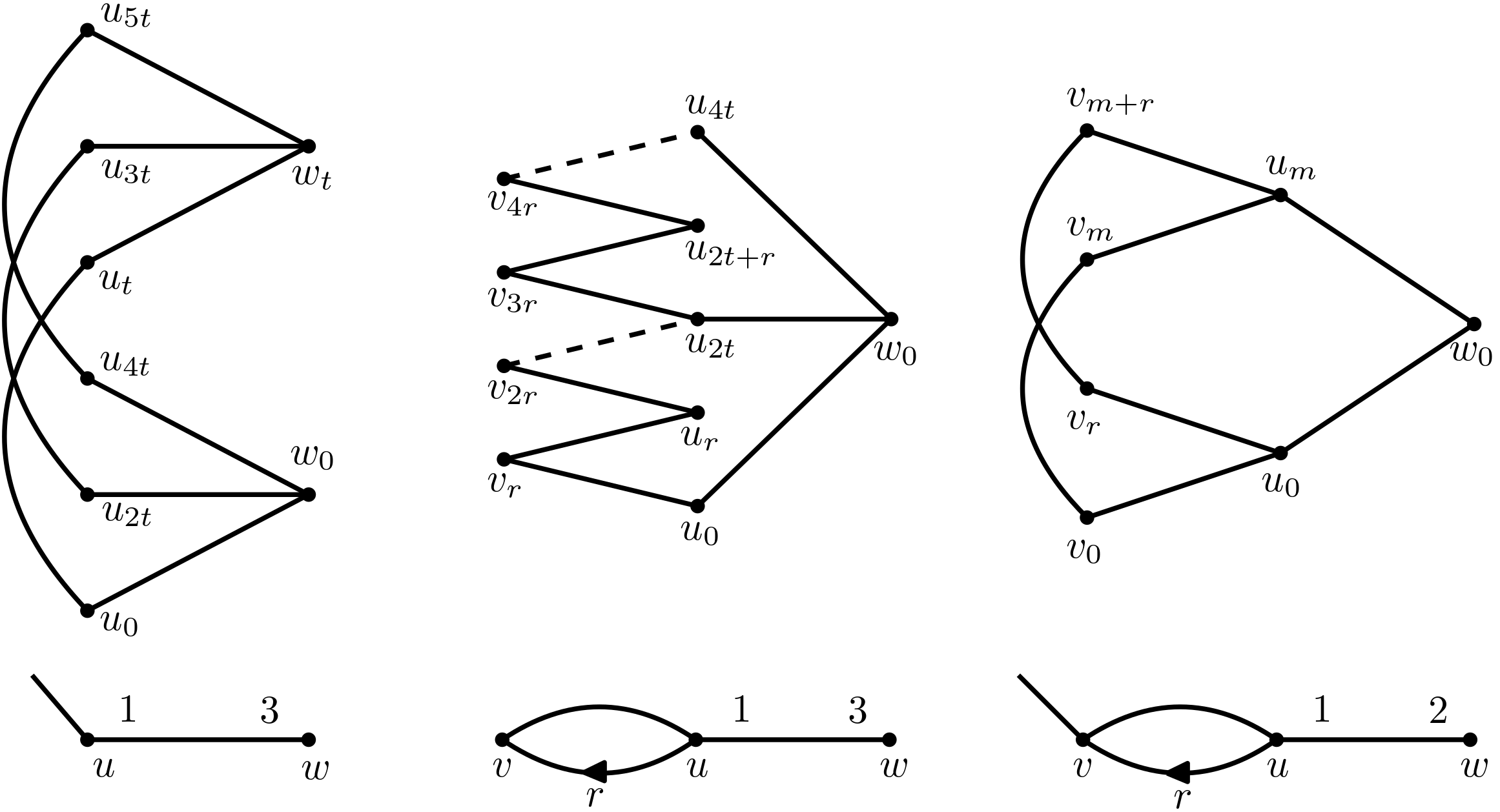}
\caption{On the top row, from left to write, a subgraph of $\Gamma_{13}(m;r)$, $\Gamma_{20}(m;r)$ and $\Gamma_{21}(m;r)$, respectively. Underneath each, the subgraph of $\Delta_i$ to which it projects.}
\label{fig:cases}
\end{figure}
\medskip

{\bf Proof of Claim \ref{cl:noVT}.}
Let $i \in I_C:=\{7,9,10,11,12,13,14,15,17,20\}$ and suppose $\Gamma$ is a vertex-transitive ccv-cover of $\Delta_i$. Since no element of $E$ is a ccv-cover of $\Delta_i$ for $i \in I_C$ (see Table \ref{tab:quo}) we see that $\Gamma \notin E$. Suppose $i \in \{7,14,15\}$. Then by Lemma \ref{lem:cycles}, the girth of $\Gamma$ is $3$ or $4$ and by Remark \ref{rem:label3}, $\Gamma$ is arc-transitive. It then follows from Lemma~\ref{lem:smallgirth1} that $\Gamma \in E$, a contradiction. If $i = 9$ then by Lemma \ref{lem:cycles}, $\Gamma$ has $3$-signature $(\epsilon_1,\epsilon_2,\epsilon_3)$, with $\epsilon_1,\epsilon_2,\epsilon_3, > 0$, and by Lemma \ref{lem:girthregular} we see that $\Gamma$ is isomorphic to $K_4$, which is a element of $E$, leading us again to a contradiction. If $i \in \{10,11\}$ then by Lemma~\ref{lem:cycles}, Remark~\ref{rem:label3} and Lemma~\ref{lem:smallgirth1} we see, once more, that $\Gamma \in E$. If $i = 17$, $\Gamma \in E$ by Lemma \ref{lem:loop}.
The cases when $i \in \{12,13,20\}$ require slightly deeper analysis. In the subsequent paragraphs, we will assume the vertices of the graphs $\Delta_i$, with $i \in \{12,13,20\}$, are named as follows: in Figure \ref{fig:quotients}, let $u$ be the leftmost vertex; $v$ the top vertex; $w$ the bottom vertex, so that $w$ is in fact the distinguished vertex of $\Delta_i$. 

Suppose $i =12$. Then $\Gamma \cong \Gamma_{12}(m;r)$ for some integers $m,r >0$ such that $m$ is even and $\gcd(m/2,r)=1$. Let $m = \iota(w)$. By Lemma \ref{lem:cycles} every vertex in $\fib(u)$ lies on a $3$-cycle. Then $v_0$ must also lie on a $3$-cycle $C$. Note that since a vertex $u_0$ is the only neighbour of $v_0$ in $\fib(u)$, all three vertices in $C$ must belong to $\fib(v)$. This is, $C=(v_0,v_r,v_{2r})$ or $C=(v_0,v_{-r},v_{-2r})$. Then $3r \equiv 0$ ($\mod 2m$). However, by item (1) of Lemma \ref{lem:simplified}, $r < \iota(\beg v) = 2m$ which implies that either $3r = 2m$ or $3r = 4m$. Moreover, since $\gcd(m/2,r) = 1$, then necessarily $r=4$ and $m=6$. The reader can verify that the graph $\Gamma_{12}(6,4)$ is not vertex-transitive.

Now, suppose $i = 13$. Then $\Gamma \cong \Gamma_{13}(m;r)$ for some integers $m,r >0$ such that $m$ is even and $\gcd(m/2,r)=1$. Let $\iota(w)= m = 2t$ and let $g$ be the girth of $\Gamma$. Note that $\Gamma$ must be arc-transitive as $\Delta_{13}$ has a $[1,3]$-edge. If $g < 6$ then by Lemma \ref{lem:smallgirth1},  $\Gamma \in E$, a contradiction. Suppose $g \geq 6$. Let $\Gamma' \leq  \Gamma$ be the subgraph induced by $\fib(u) \cup \fib(w)$. It is plain to see that $\Gamma'$ contains a $6$-cycle (see Figure \ref{fig:cases}, left) and thus $g=6$. Moreover, every edge incident to $w_i$, $i \in \ZZ_{\iota(w)}$, lies on two distinct $6$-cycles of $\Gamma'$. The edge $u_0u_{3t}$ also lies on two distinct $6$-cycles, both contained in $\Gamma'$. Since $\Gamma_{13}(m;r)$ is arc-transitive, the edge $u_0v_0$ must lie on a $6$-cycle $C$. However, $C$ must necessarily visit either $u_0w_0$ or $u_0u_{3t}$. This is, one of $u_0w_0$ or $u_0u_{3t}$ lies on three distinct $6$-cycles. It follows that the $6$-signature of $\Gamma$ is not $(2,2,2)$ and by Lemma \ref{lem:smallgirth2}, $\Gamma \in E$, a contradiction.

Finally, suppose $i = 20$ and $\Gamma$ is isomorphic to $\Gamma_{20}(m;r)$ where $m$ is even, $r \neq 0$ and $\gcd(m/2,r)=1$. Let $\iota(w)=m=2t$ and let $g$ be the girth of $\Gamma$. Observe that $\Gamma$ must be arc-transitive since $\Delta_{20}$ has a $[1,3]$-edge. If $g < 6$, then by Lemma \ref{lem:smallgirth2}, $\Gamma \in E$, a contradiction. Suppose $g \geq 6$ and see that if $r \equiv 0$ ($\mod 2t$) then $(w_0,u_0,v_r,u_r)$ is a $4$-cycle, contradicting our assumption that $g \geq 6$. Hence we may assume $r \not \equiv 0$ ($\mod 2t$). Now, note that $\iota(v) = \iota(u) = 3 \cdot \iota(w) = 6t$. Then the voltage of the semi-edge at $v$ must be $\iota(v) /2 = 3t$. This implies that each $v_i$ is adjacent to $v_{i+3t}$, and in particular that $(u_0,v_0,v_{3t},u_{3t},v_{3t+r},v_r)$ is a $6$-cycle of $\Gamma_{20}(m;r)$. Then $w_0$ must also lie on a $6$-cycle $C$. Since both $\fib(w)$ and $\fib(u)$ are independent sets, if a cycle visits $\fib(w)$ twice, it must have length at least $8$. It follows that $w_0$ is the only vertex in $C$ belonging to $\fib(w)$. It is straightforward to see that a $6$-cycle through $w_0$ must necessarily be $C_i=(w_0,u_{it},v_{it+r},u_{it+r},v_{it+2r},u_{it+2r})$, for some $i \in \{0,2,4\}$. This is, in Figure \ref{fig:cases} (middle), the dotted lines must be edges. This implies that $2r \equiv 0$ ($\mod 2t$). Moreover, since $0 < r < 3t$ we see that $2r = 2t$. Then the edge $u_0v_0$ lies on both $C_0$ and $C_4$, which are distinct cycles since $r \not \equiv 0$ ($\mod 2t$). Since $u_0v_0$ also lies in $C$, it follows from Lemma \ref{lem:smallgirth2} that $\Gamma \in E$, once more a contradiction.  
This completes the proof of Claim \ref{cl:noVT}.
\medskip

{\bf Proof of Theorem~\ref{the:vt}.}
Let $\Gamma$ be a simple, connected, cubic graph admitting a cyclic group of automorphism $G$ with at most $3$ orbits on vertices. Suppose $\Gamma$ is vertex-transitive. Then, $\Gamma$ is a ccv-cover of $\Delta_i$ for some $i \in \{1,\ldots,25\}$. By Claim \ref{cl:noVT}, $i \notin I_C$. If $i \in I_E$, then by Claim \ref{cl:ex}, $\Gamma$ must be one of the exceptional graphs in $E$. However, every graph in $E$ is either isomorphic to the Tutte-Coxeter graph or belongs to one of the five infinite families described in Theorem \ref{the:vt}. Indeed, $K_4 \cong \Gamma_i(4;1)$, $K_{3,3}\cong\Gamma_4(3;1,2)$, $Q_3\cong\Gamma_2(4;1,1)$, $\textrm{GP}(5,2)\cong\Gamma_2(5;2,1)$, the Heawood graph is isomorphic to $\Gamma_4(7;1,3)$, $\textrm{GP}(8,3)\cong\Gamma_2(8;3,1)$, the Pappus graph is isomorphic to $\Gamma_{22}(6;2,1)$ and $\textrm{GP}(10,3)\cong\Gamma_2(10;3,1)$. It remains to see what happens when $i \in I_M=\{1,2,3,4,22,23,24,25\}$.

If $i = 1$, then $\Gamma \cong \Gamma_1(m;r)$ for some integers $m$ and $r$ where $m \geq 4$, $m$ is even, $r \neq 0$ and $\gcd(\frac{m}{2},r)$ (see Theorem~\ref{the:main}). Observe that $\gcd(m,r) \in \{1,2\}$. If $\gcd(m,r) = 1$ then $r$ is a unit in the ring $\ZZ_m$. That is, $r$ has a multiplicative inverse $r^{-1}$ in $\ZZ_m$, and $\gcd(m,r^{-1})=1$. Then by Lemma~\ref{lem:coprimemulti}, $\Gamma_1(m;r) \cong \Gamma_1(m;1)$. On the other hand, if $\gcd(m,r)=2$, then $\frac{r}{2}$ is a unit in $\ZZ_m$ and, once again by Lemma~\ref{lem:coprimemulti}, $\Gamma_1(m;r) \cong \Gamma_1(m;2)$.

Suppose $i \in \{2,3,4\}$. Note that for integers $m$, $r$ and $s$, the graphs $\Gamma_2(m;r,s)$ and $\Gamma_3(m;r,s)$ correspond, respectively, to the graphs $I(m,r,s)$ and $H(m,r,s)$ described in \cite{bic}. It then follows from Proposition 4 of \cite{bic} and Theorem~\ref{the:main}, that if $i =2$ then $\Gamma \cong \Gamma_2(m;r,1)$, for some $m \geq 3$ and $r^2 \equiv \pm1$ ($\mod m$), or $m=10$ and $r=2$. By the same token, if $i=4$ then $\Gamma \cong \Gamma_4(m;r,s)$ where $m \geq 3$, $r \neq s$ and $\gcd(m,r,s)=1$. If $i=3$ then by \cite[Section 3]{bic}, $\Gamma$ is a circulant and so is also a ccv-cover of $\Delta_1$. 

Now, suppose $i \in \{22,23,24,25\}$. Observe that the graphs $\Gamma_{22}(m;r,s)$, $\Gamma_{23}(m;r,s)$, $\Gamma_{24}(m;r,s)$ and $\Gamma_{25}(m;r,s)$ correspond to the graphs $T_1(\frac{m}{2},r,s)$, $T_2(\frac{m}{2},r,s)$, $T_3(\frac{m}{2},r)$ and $T_4(\frac{m}{2},r,s)$ described in \cite{MPtricirc}. By Theorem $1$ of \cite{MPtricirc}, if $i = 22$ then $\Gamma$ is isomorphic to $\Gamma_{22}(m;2,1)$ with $\frac{m}{2}$ odd. By the same theorem if $i = 23$ then $\Gamma \cong \Gamma_{23}(m;r,1)$ with $\frac{m}{2} \equiv 1$ $(\mod 4)$ and $r = (\frac{m}{2}+3)/2$,  or $\frac{m}{2} \equiv 3$ $(\mod 4)$ and $r = (\frac{3m}{2}+3)/2$, or $m=2$ and $r=0$. If $i = 24$ then by Theorem~50 of \cite{MPtricirc} $\Gamma$ is also a cover of $\Delta_1$. Finally if $i = 25$ then by  Theorem 52 of \cite{MPtricirc} $\Gamma$ has order smaller than $54$ and by Section 3 of \cite{MPtricirc}, $\Gamma$ isomorphic to Tutte-Coxeter graph (referred to as Tutte's 8-cage in \cite{MPtricirc}).

For the converse, suppose $\Gamma$ is the Tutte-Coxeter graph or belongs to one of the $5$ infinite families described in Theorem~\ref{the:vt}. It is well known that the Tutte-Coxeter graph is arc-transitive and thus also vertex-transitive. A cover of $\Delta_1$ must be vertex-transitive as it has a cyclic group of automorphism with a single orbit on vertices. It is shown in Theorem 4 of \cite{bic}, that the graphs belonging to the families of items (2) and (3) of Theorem~\ref{the:vt} are vertex-transitive. Finally, the graphs in items (4) and (5) of Theorem~\ref{the:vt} are vertex-transitive by Theorem~1 of \cite{MPtricirc}. This completes the proof.

\end{document}